\documentclass{amsart}

\usepackage{amsrefs}

\usepackage{amssymb, amsmath, amsfonts, stmaryrd,  amsthm, verbatim, cite, enumerate, mathtools, color}
\usepackage{graphicx}
\usepackage{epstopdf}
\usepackage{relsize}

\newcommand{\mychar}[1]{%
\raisebox{-4pt}{
  \begingroup\normalfont
  \includegraphics[height=0.6cm]{#1}%
  \endgroup}
}

\newcommand{\R}{\mathbb{R}}

\newcommand{\Z}{\mathbb{Z}}

\newcommand{\N}{\mathbb{N}}

\newcommand{\Ca}{\mathcal{C}}
\newcommand{\D}{\mathcal{D}}
\newcommand{\Web}{\mathbf{Web}}

\newcommand{\Ob}{\operatorname{Ob}}
\newcommand{\Hom}{\operatorname{Hom}}
\newcommand{\Mor}{\operatorname{Mor}}
\renewcommand{\ker}{\operatorname{ker}}
\newcommand{\Dom}{\operatorname{Dom}}

\numberwithin{equation}{section}

\theoremstyle{plain} 
\newtheorem{thm}[equation]{Theorem}

\newtheorem{lem}[equation]{Lemma}

\theoremstyle{definition}
\newtheorem{defn}[equation]{Definition}

\theoremstyle{remark}
\newtheorem{rem}[equation]{Remark}
\newtheorem{ex}[equation]{Example}

\begin{document}

\title[An extension of the quantum $sl(3)$-invariant to tangles]{A presentation for the category of $sl(3)$ webs and an extension of the quantum $sl(3)$-invariant to tangles}

\author{Nipun Amarasinghe} 
\address{Department of Mathematics, California State University, Fresno, CA 93740, USA}


\email{NipunVAmarasinghe@gmail.com}





\begin{abstract}
We extend the $sl(3)$-polynomial invariant for links to tangles. Motivated by Kuperberg's construction of this invariant via planar trivalent graphs, we first define a category of $sl(3)$ webs and its sister linear category, and describe them via generators and relations. Then we define a functor from the category of oriented tangles to the linear category of $sl(3)$ webs, which yields an invariant for tangles and allows us to recover the $sl(3)$-polynomial invariant for links.
\end{abstract}

\maketitle


\section{Introduction}
\label{sec:intro}

For any simple Lie algebra $\mathfrak{g}$, the representation theory of its quantized universal enveloping algebra, $U_q(\mathfrak{g})$, gives us an invariant for links that can be taken to be $\Z[q, q^{-1}]$-valued; for details, we refer the reader to the work by Reshetikhin and Turaev in~\cite{Resh}. In the case of the Lie algebra $\mathfrak{sl}_3(\mathbb{C})$, which we will just denote as $sl(3)$, we get the $sl(3)$-polynomial invariant for links through this method. In~\cite{Kup}, Kuperberg uses a calculus of planar trivalent graphs, also called \textit{closed webs}, to both describe the representation theory of $U_q(sl(3))$ and come up with a way to calculate the $sl(3)$-polynomial of a given link through combinatorial methods. This will be summarized in greater detail in Section~\ref{sl3link}, where we use a normalization of Kuperberg's invariant given by Khovanov in~\cite{Khov} and describe the $sl(3)$-polynomial invariant of a link as a state sum evaluation of link flattenings, which are closed $sl(3)$ webs.

 An \textit{oriented tangle} is a smooth embedding in $\R^2 \times [0,1]$ of a finite disjoint union of circles and intervals with orientation, such that the manifold boundary of this embedding is a subset of the lines $\R\times \{0\} \times \{0\}$ and $\R \times \{0\} \times \{1\}$ (with some extra conditions). For each tangle, we can define a \textit{source sequence} and \textit{target sequence} which are finite strings of the symbols $``+"$ and $``-"$, that contain information about the boundary points of a tangle on the lines $\R\times \{0\} \times \{0\}$ and $\R \times \{0\} \times \{1\}$, respectively. We can then define a strict tensor category of oriented tangles $\mathbf{OTa}$ by thinking of each tangle as a morphism from its source sequence to its target sequence. In $\mathbf{OTa}$, the composition of two composable tangles $L_1\circ L_2$ is the tangle acquired by stacking $L_1$ on top of $L_2$ and the tensor product of any two tangles $L_1 \otimes L_2$ is the tangle with $L_2$ to the right of $L_1$. For more information about the category of oriented tangles, we refer the reader to Turaev's work in~\cite{Tur}. The motivation of this paper is to use methods from the theory of strict tensor categories to extend the $sl(3)$-invariant to tangles and describe it using only category theory.

Motivated by the approach given by East in~\cite{East}, in Section~\ref{STC} we define generators and relations for such categories  via algebraic methods using free and quotient tensor categories. In Section~\ref{webdefine}, we define a web to be an oriented graph in $S:= R \times [0,1]$ with vertices that are either sinks or sources, where the vertices in $S^\circ$ are trivalent and vertices in $\partial S$ are univalent (with some extra conditions). Similar to the case of tangles, we may enrich this set of webs with the structure of a strict tensor category by defining the composition and tensor product of webs. This acts as a categorical generalization of the closed webs described by Kuperberg in~\cite{Kup}, analogous to how tangles are a categorical generalization of links. In Section~\ref{webpres} we utilize the tools given in Section~\ref{STC} to give a presentation of the category of webs, \textbf{Web}, in terms of generators and relations. In Section~\ref{skeinmod}, we then use this presentation to describe a skein module of formal linear combinations of webs modulo relations given by the Kuperberg bracket, a closed web invariant used in Kuperberg's construction of the $sl(3)$-polynomial invariant of links. Elements from these modules then naturally form a strict tensor category, $\textbf{LWeb}$, the linear web category. Finally, using a presentation for the category $\textbf{OTa}$ given in~\cite{Tur}, we define an invariant for tangles via a functor $F: \textbf{OTa} \to \textbf{LWeb}$, given by a specific assignment on the generators of $\mathbf{OTa}$. If $L \in \textbf{OTa}$ is a link, that is $L$ is a morphism from the empty string to the emtpy string, $F(L)$ is an equivalence class of linear  $\Z[q, q^{-1}]$-combinations of closed webs that contains a unique element of $\Z[q, q^{-1}]$, which is the $sl(3)$ polynomial invariant for $L$.

 The target audience for this paper are young researchers interested in categorical methods in knot theory. For this reason, we have chosen to state and prove many basic known results regarding strict tensor categories and their presentations in order to both keep this paper as self contained as possible and to provide proofs in greater detail than in the standard references. We only assume basic knowledge of topology and category theory. 

\section{Strict tensor categories} \label{STC}

Ultimately, we wish to describe an invariant $F$ of oriented tangles. In order to define such an invariant neatly, we use a strict tensor category structure on the set of homotopy classes of oriented tangles, denoted $\mathbf{OTa}$, and define $F$ as a functor from $\mathbf{OTa}$ that can be uniquely determined by where $F$ sends generators of $\mathbf{OTa}$. We now provide the necessary background on strict tensor categories and their presentations so that we may be able to give a presention of the codomain of $F$ and justify defining $F$ uniquely via an appropriate map of generators.

Let $\mathcal{C}$ be a category and let $f: X \to Y$ be a morphism in $\mathcal{C}$. We will use the following notation:
\begin{itemize}
\item $\Ob(\mathcal{C})$ is the class of objects in $\mathcal{C}$
\item $\Mor(\mathcal{C})$ is the class of morphisms in $\mathcal{C}$
\item $\Hom_\mathcal{C}(X, Y)$ is the class of morphisms from $X$ to $Y$
\item $\mathbf{s}(f) = X$ is the \textit{source} of $f$
\item $\mathbf{t}(f) = Y$ is the \textit{target} of $f$ .
\end{itemize}

For the purposes of this paper, every category we consider is small, so we assume that $\Ob(\mathcal{C})$, $\Mor(\mathcal{C})$, and $\Hom_\mathcal{C}(X, Y)$ are all set sized. 
\begin{defn}
A \textit{strict tensor category} (or a \textit{strict monoidal category}) is a triple $(\mathcal{C},\otimes ,O)$ consisting of a category $\mathcal{C}$, a covariant bifunctor $\otimes: \Ca \times \Ca \to \Ca$ where we denote $\otimes (a, b) = a \otimes b$, and an object $O \in \Ob(\Ca)$, such that the following hold for all $f, g, h \in \Mor(\Ca)$:
\begin{itemize}
\item $f\otimes (g \otimes h) = (f\otimes g) \otimes h$
\item $f \otimes Id_O = f = Id_O \otimes f.$
\end{itemize}
Note that the above equalities imply that the relations $X\otimes (Y \otimes Z) = (X\otimes Y) \otimes Z$ and $X \otimes O = X = O \otimes X$ also hold for all objects $X, Y, Z \in \Ob(\mathcal{C})$.
\end{defn}
Although there is a more general notion of a tensor category, unless otherwise specified, we use the term tensor category to refer to a strict tensor category. With the definition of a tensor category at hand, we consider now functors between tensor categories, which will allow us to create a category of tensor categories.

\begin{defn}
Let $S$ be any set. We define $S-\mathbf{Cat^\otimes}$ as the small category whose objects are tensor categories $\Ca$, such that $\Ob(\Ca) = S$, and whose morphisms are tensor-product-preserving functors that act as the identity on objects. To be specific, for any morphism  $\phi: (\Ca, \otimes_{\Ca}, O_\Ca) \to (\D, \otimes_{\D}, O_\D)$ in $S-\mathbf{Cat^\otimes}$  and any $f, g \in \Mor(\Ca)$, the following holds: 
\[ \phi(f \otimes_{\Ca} g) = \phi(f) \otimes_\D \phi(g).\]
\end{defn}

We note that $S-\mathbf{Cat^\otimes}$ is a subcategory of $\mathbf{Cat}$, the category of all small categories. The category $S-\mathbf{Cat^\otimes}$ could be defined as a more general category whose objects are (strict) tensor categories by not requiring that all the categories have the same object set and that the functors are the identity on objects. However, adding these restrictions will suit our purposes better, as it does make the category simpler and it has nice properties that a more general category may not have. In particular, the following hold in this category. 

\begin{itemize}
\item Given a morphism $\phi: \Ca \to \D$ in $S-\mathbf{Cat^\otimes}$, the image of $\phi$, $\phi(\Ca)$, forms a tensor subcategory of $\D$ and is in $S-\mathbf{Cat^\otimes}$.
\item Given a morphism $\phi: \Ca \to \D$ in $S-\mathbf{Cat^\otimes}$, if $\phi$ is a bijective map for the sets $\Mor(\Ca)\to \Mor(D)$, then $\phi$ is an isomorphism in $S-\mathbf{Cat^\otimes}$.
\end{itemize} 

 It is to our interest to construct a presentation of the tensor category $\mathbf{Web}$ via generators and relations, so we must first define what such a presentation entails in a generic tensor category.

\begin{defn} \label{RelCat}
Given a tensor category $\Ca$, a \textit{tensor congruence} on $\Ca$, denoted $R$, is an equivalence relation on $\Mor(\Ca)$ satisfying the following properties:
\begin{enumerate}
\item[(i)] For any $(f,g) \in R$, $f$ and $g$ must have the same source and target.
\item[(ii)]For any $(f,g),(h,k) \in R$, $(f\circ h, g\circ k)\in R$, if these compositions exist. 
\item[(iii)]For any $(f,g),(h ,k) \in R$, $(f \otimes h, g \otimes k)\in R$.
\end{enumerate}
\end{defn}

\begin{rem}
Note that tensor congruences have the property that for any collection of tensor congruences $R_\alpha \subseteq \Mor(\Ca) \times \Mor(\Ca)$ for $\alpha \in I$, where $I$ is an indexing set, then $\bigcap_{\alpha \in A} R_\alpha$ is a tensor congruence. Thus, given any subset $A$ of $\Mor(\Ca) \times \Mor(\Ca)$ that satisfies property (i) of a tensor congruence in Definition~\ref{RelCat}, there exists a minimal tensor congruence that contains $A$. We will refer to this tensor congruence as the \textit{tensor congruence generated by $A$}, and denote it by $\langle A \rangle$.
\end{rem}
\begin{rem}
	For a subset $A \subset \Mor(\Ca) \times \Mor(\Ca)$ satisfiying property (i) of a tensor congruence in Definition~\ref{RelCat}, we may use the notation $f \sim g$ for an element $(f, g ) \in A$ as this notation looks more natural when we eventually start using tensor congruences to quotient and present categories. 
\end{rem}

\begin{ex} \label{ker}
Given a morphism $\phi: \Ca \to \D$ in $S-\mathbf{Cat^\otimes}$, one tensor congruence that is particularly important is
$$\ker(\phi) := \{(f,g) \in \Mor(\Ca) \times \Mor(\Ca): \phi(f) = \phi(g)\}.$$
We now check that $\ker(\phi)$ satisfies Definition~\ref{RelCat}.
\begin{itemize}
\item[(i)] Since $\phi$ must be the identity on objects, $\phi(f) = \phi(g)$ implies that $$\mathbf{s}(f) = \mathbf{s}(\phi(f)) = \mathbf{s}(\phi(g)) = \mathbf{s}(g),$$ 
and likewise $$\mathbf{t}(f) = \mathbf{t}(\phi(f)) = \mathbf{t}(\phi(g)) = \mathbf{t}(g).$$
\item[(ii)]For any $(f,g),(h,k) \in \ker(\phi)$ such that $f \circ h$ and $g \circ k$ exist, we have that $$\phi( f \circ h) = \phi(f) \circ \phi(h)  = \phi(g) \circ \phi(k)= \phi(g \circ k),$$ 
and thus $(f \circ h, g \circ k)\in \ker(\phi).$
\item[(iii)]For any $(f,g),(h,k) \in \ker(\phi)$, $$\phi(f \otimes h)= \phi(f)\otimes \phi(h)=  \phi(g)\otimes \phi(k) =\phi( g \otimes k),$$ so $(f \otimes h, g \otimes k)\in \ker(\phi).$
\end{itemize} 
Therefore, $\ker(\phi)$ is a tensor congruence on the category $\Ca$.
\end{ex}

\begin{defn}
Given a tensor category $(\mathcal{C}, \otimes, O_\Ca)$ and a tensor congruence $R$ on $\Ca$, the \textit{quotient category}, $\Ca/R$, is the tensor category ($\Ca/R, \otimes, O_{\Ca/R})$ consisting of the following data:
\begin{itemize}
\item \textbf{Objects:} $Ob(\Ca/R) = Ob(\Ca)$.
\item \textbf{Morphisms ($X \to Y$):} These are equivalence classes under $R$ of morphisms in $Hom_\Ca(X, Y)$. Given an $f\in Mor(\Ca)$, we use $\overline{f}$ to denote the equivalence class of $f$ under $R$.
\item \textbf{Composition:} $\overline{f} \circ \overline{g} = \overline{f\circ g}$. This is well-defined by property (ii) of a tensor congruence.
\item \textbf{Tensor Product:} $\overline{f} \otimes \overline{g} = \overline{f \otimes g}$. This is well-defined by property (iii) of a tensor congruence.
\item \textbf{Identity ($X$):} $Id_{X} = \overline{Id_X}$.
\item \textbf{Tensor Unit:} $O_{\Ca/R} = \overline{O_\Ca}$.
\end{itemize}
\end{defn}

 This terminology for the quotient of a tensor category $\Ca$ and the kernel of a functor whose domain is $\Ca$ will become relevant in the context of our next theorem. The relation between the two concepts for categories is similar to the relation between quotients and kernels in the context of groups, rings, and modules, as it also yields a first isomorphism theorem of sorts. 

\begin{thm} \label{fit}
For any morphism $\phi: \Ca \to \D$ in the category $S-\mathbf{Cat^\otimes}$, the following holds:
$$\Ca/\ker(\phi) \cong \phi(\Ca).$$
\begin{proof}
We claim that the map $\overline{\phi}: \Ca/\ker(\phi) \to \phi(\Ca)$ defined on morphisms by
 $$\overline{\phi}: \Mor(\Ca/\ker(\phi)) \to \Mor(\phi(\Ca)), \text{where} \,\, \overline{f} \mapsto \phi(f) \text{ for all } \overline{f} \in \Mor(\Ca/\ker(\phi)),$$ 
 and defined as the identity on objects, is an isomorphism in $S-\mathbf{Cat^\otimes}$. 
 
 To show that this is indeed a morphism in $S-\mathbf{Cat^\otimes}$, we note that the map is well defined, since $\overline{f}= \overline{g}$ if and only if $\phi(f) = \phi(g)$.
 Moreover, by the definition of composition, tensor product, and identity in the quotient category, we see that this map is indeed a tensor-product-preserving functor that is the identity on objects.
 
  To show that this morphism is an isomorphism in $S-\mathbf{Cat^\otimes}$, it suffices to show that the function $\overline{\phi}: \Mor(\Ca/\ker(\phi)) \to \Mor(\phi(\Ca))$ is a bijection.
   Surely, this is the case since for any $\phi(f) \in \Mor(\phi(\Ca))$, $\overline{\phi}(\overline{f}) = \phi(f)$ 
   and whenever $\phi(f) = \phi(g)$, we must have $\overline{f}= \overline{g}$.
    Therefore, $\overline{\phi}: \Ca/\ker(\phi) \to \phi(\Ca)$ is an isomorphism in $S-\mathbf{Cat^\otimes}$.
\end{proof}
\end{thm}

Now we wish to define a notion of a free tensor category. When working with general small categories, we define the free category generated by a small directed multi-graph $G$ to be the category with objects the vertices of $G$ and with morphisms the strings of composable edges, including empty strings on a vertex, where the composition is concatenation of these strings; see~\cite{CWM}. However, an arbitrary formal expression involving both composition and tensor product takes a little more work to define. 

\begin{defn} \label{alphabet}
Let $G$ be a small directed multi-graph, where by small we mean that we only require the number of vertices and edges in this graph to be set sized. Additionally, suppose that the vertex set of $G$ has a monoid structure $(V(G), \otimes, \varnothing)$. Then the \textit{words in the alphabet of $G$}, denoted $A(G)$, is the set of of formal expressions equipped with source and target functions $\mathbf{s}$, $\mathbf{t} : A(G)\to V(G)$, constructed only by requiring the following.
\begin{itemize}
\item[(A1)] For any vertex $v$, the empty path on $v$, denoted by $\iota_v$, is in $A(G)$ and satisfies $\mathbf{s}(\iota_v)= \mathbf{t}(\iota_v)=v$.
\item[(A2)] For any edge $e$ from $v_1$ to $v_2$,  $e$ is in $A(G)$ with $\mathbf{s}(e)= v_1$ and $\mathbf{t}(e)=v_2$.
\item[(A3)] For any two words $a$ and $b$ with $\mathbf{t}(a) = \mathbf{s}(b)$, the formal expression $(b \circ a)$ is in $A(G)$, with $\mathbf{s}((b \circ a)) = \mathbf{s}(a)$ and $\mathbf{t}((b \circ a)) = \mathbf{t}(b)$.
\item[(A4)] For any two words $a$ and $b$, the formal expression $(a \otimes b)$ is in $A(G)$, with $\mathbf{s}((a \otimes b)) = \mathbf{s}(a)\otimes \mathbf{s}(b) $ and $\mathbf{t}((a \otimes b)) = \mathbf{t}(a)\otimes \mathbf{t}(b) $.
\end{itemize}
More explicitly, the elements of $A(G)$ are all formal expressions that are constructed by applying a finite sequence of the `moves' (A3) and (A4) on the elements required by (A1) and (A2). In addition to this, to any element of $A(G)$ we assign a natural number, called the word's \textit{rank}. We do this inductively, where all words of rank 1 are exactly those elements of $A(G)$ given by (A1) and (A2) above. Then, given all words of rank less than or equal to some $n \in \N$, we define all words of rank at most $n+1$ as the words formed by applying the construction in (A3) or (A4) to words of rank at most $n$.

Similarly, we also inductively define the \textit{collection of subwords} in a word in the alphabet $A(G)$. If a word is of rank 1, then the only subword is itself. However, having defined subwords for all words of rank up to some $n \in \N$, then any word $w$ of rank $n+1$ is of the form $a \circ b$ or $a \otimes b$, where $a$ and $b$ are words of rank at most $n$. Then we define the collection of subwords of the word $w$ to be the word itself along with all subwords of $a$ and all subwords of $b$. 	
\end{defn}

\begin{rem}
	Note that because of the parenthesis in the moves (A3) and (A4) moves are part of the formal expression. Thus, we never have associativity of composition or tensor product. However, one consequence of this is that if $w \in A(G)$ and $w = a \otimes b = a' \otimes b'$ for $a, b, a', b' \in A(G)$ then $a = a'$ and $b= b'$. A similar statement holds for composition.
\end{rem}

The following definition will give us notation to deal with elements of $A(G)$ better and define what it means to replace a subword. 
\begin{defn}
	Let $w, a \in A(G)$. A \textit{word decomposition of $w$ starting at $a$} is a finite sequence (possibly empty) of functions $\{f_k: S_k\subseteq A(G) \to A(G)\}_{k=1}^n$ such that 
	$f_k \circ f_{k-1} \circ \dots \circ f_1(a) \in S_{k+1}$,  for each $k \in \{1, \dots , n\},    w = f_n \circ f_{n-1} \circ \cdots \circ f_1(a),$ and each $f_k$ ($1 \leq k \leq n$) is one of the following functions:
	\begin{itemize}
		\item[(i)] $A(3)^L_b: \{c \in A(G): \mathbf{t}(c) = \mathbf{s}(b)\} \to A(G)$; $c \mapsto b  \circ c$\\
		\item[(ii)] $A(3)^R_b: \{c \in A(G): \mathbf{t}(b) = \mathbf{s}(c)\} \to A(G)$; $c \mapsto c \circ b$ \\
		\item[(iii)] $A(4)^L_b: A(G) \to A(G)$; $c \mapsto b \otimes c$\\
		\item[(iv)] $A(4)^R_b: A(G) \to A(G)$; $c \mapsto c \otimes b.$
	\end{itemize}
\end{defn}

\begin{rem}
	By the definition of $A(G)$, for any $w \in A(G)$, there exists a word $a$ of rank 1 and a decomposition of $w$ starting at $a$. Similarly, for any $w \in A(G)$, the empty sequence is a word decomposition of $w$ starting at $w$.
\end{rem}

 One use for word decompositions is that it gives us an equivalent definition of the subwords of a given word $w \in A(G)$.
\begin{lem}
	Let $w \in A(G)$. A word $a \in A(G)$ is a subword of $w$ if and only if there exists a word decomposition of $w$ starting at $a$. 
\end{lem}
\begin{proof}
	Suppose that $a$ is a subword of $w$, we show that there is a word decomposition of $w$ starting at $a$ by induction on the rank of $w$. If $w$ is rank $1$, then $a = w$ and so the empty sequence is a word decomposition of $w$ starting at $a$. Now suppose that for every word $w'$ of rank less than $n$ and for any subword $a'$ of $w'$, there is a word decomposition of $w'$ starting at $a'$. Suppose $w$ is of rank $n$. Then there exists $b,c \in A(G)$ of rank less than $n$ such that either $w = b \otimes c$ or $w = b \circ c$. In the case that $w = b \circ c$, suppose without loss of generality that $a$ is a subword of $b$. Then there exists a word decomposition of $b$ starting at $a$, say $\{f_k\}_{k=1}^n$. If we define $f_{n+1} := A(3)_c^R$, we have 
	$$f_{n+1} \circ f_n \circ \cdots f_1(a) = f_{n+1}(b) = b \circ c = w,$$ 
and thus $\{f_k\}_{k=1}^{n+1}$ is a word decomposition of $w$ starting at $a$. The case for $w = b \otimes c$ is similar. This concludes the induction. 
	
	Conversely, suppose that there exists a word decomposition of $w$ starting at $a$, say $\{f_k\}_{k=1}^n$. Then define 
	$$a_k = \begin{cases}
		f_k \circ f_{k-1} \circ \cdots \circ f_1(a), & \text{if } k \in \{1,2, \cdots n\}\\
		a& \text{if } k=0 
		\end{cases} $$ 
	so that $a_0 = a$ and $a_n = w$. Then regardless of whether $f_{k}$ is of the form $A(3)^R_c$, $A(3)^L_c$, $A(4)^R_c,$ or $A(4)^L_c$, $a_{k-1}$ is a subword of $a_{k} = f_k(a_{k-1})$, for all $k \in \{1, \cdots, n\}$. Therefore, by transitivity of the subword relation, $a_0 = a$ is a subword of $a_n = w$. This concludes the proof.
\end{proof}

The set $A(G)$ forms a good start for defining the morphisms of free tensor categories. However, we want these words to satisfy the properties of the product and composition in a tensor category. To endow these words with those equalities, we equip $A(G)$ with the following equivalence relation.
\begin{defn}\label{replace}
Let $w_1, w_2 \in A(G)$, $a_1$ a subword of $w_1$ and $a_2$ a subword of $w_2$. Then we say $w_2$ is obtained from $w_1$ by \textit{replacing the subword $a_1$ with $a_2$} if and only if there exists a word decomposition of $w_1$, $\{f_k\}_{k=1}^n$, starting at $a_1$ such that $\{f_k\}_{k=1}^n$ is also a word decomposition of $w_2$ starting at $a_2$. 
\end{defn} 

\begin{defn}\label{EquivWords}
Two words, $w_1$ and $w_2$, in $A(G)$ are said to be \textit{equivalent} if and only if there exists a finite sequence of words $w_1 = d_1, d_2, \ldots ,d_n = w_2$, such that each $d_i$ is obtained from $d_{i-1}$ by replacing a subword $a_1$ with $a_2$, such that $\{a_1, a_2\}$ is of one of the following forms
\begin{enumerate}
\item $\{(a \circ b) \circ c, a \circ (b \circ c)\},$
\item $ \{a \circ \iota_{\mathbf{s}(a)} ,  a\} $ or $\{a,\iota_{\mathbf{t}(a)} \circ a \},$
\item $\{(a \circ a') \otimes (b \circ b')  ,  (a \otimes b) \circ (a' \otimes b')\},$
\item $\{(\iota_{v_1}) \otimes (\iota_{v_2})  ,  \iota_{v_1\otimes v_2}\},$
\item $\{(a \otimes b) \otimes c, a \otimes (b \otimes c)\},$
\item $\{a \otimes \iota_\varnothing ,  a\}$ or $\{a,  \iota_\varnothing \otimes a\}.$
\end{enumerate}
\end{defn}

\begin{rem}
Equivalence of words is an equivalence relation on the words of the alphabet $A(G)$. 
\end{rem} 

\begin{rem}
	Due to relations (1) and (5) in Definition~\ref{EquivWords}, if $a,b,c \in A(G)$ then we may use $a \otimes b \otimes c$ to denote either $(a \otimes b) \otimes c$ or $a \otimes (b \otimes c)$ and $a \circ b \circ c$ to for either $(a \circ b) \circ c$ or $a \circ (b \circ c)$, if it exists. 
\end{rem}
The following lemma is a restatement of Lemma 2.3.3 in ~\cite{Tur} and will later help in dealing with arbitrary elements in $A(G)$ up to equivalence.
\begin{lem}\label{generalform}
Any word of an alphabet is equivalent to a word of the form
\[(\iota_{v_1}\otimes a_1 \otimes \iota_{u_1})\circ (\iota_{v_2}\otimes a_2 \otimes \iota_{u_2}) \circ \ldots \circ (\iota_{v_n}\otimes a_n \otimes \iota_{u_n}),\]
for some $n \in \mathbb{N}$, where $a_i$ are words of rank $1$, and $v_i$ and $u_i$ are vertices, for all $1\leq i \leq n$.
\end{lem}
\begin{proof}
We provide a proof by induction on the rank of the word. If a word $w$ is of rank $1,$ then we may write 
$$ w = \iota_\varnothing \otimes w \otimes \iota_\varnothing.$$
We now assume that the claim has been established for words up to rank $k$. If a word $w$ is of rank $k+1$, then we may write $w = a\circ b$ or $w = a \otimes b$, for some words $a$ and $b$ of rank at most $k$. If $w = a \circ b$, then we may rewrite
$$a = (\iota_{v_1}\otimes a_1 \otimes \iota_{u_1})\circ (\iota_{v_2}\otimes a_2 \otimes \iota_{u_2}) \circ \ldots \circ (\iota_{v_n}\otimes a_n \otimes \iota_{u_n})$$
and 
$$b = (\iota_{w_1}\otimes b_1 \otimes \iota_{x_1})\circ (\iota_{w_2}\otimes b_2 \otimes \iota_{x_2}) \circ \ldots \circ (\iota_{w_m}\otimes b_m \otimes \iota_{x_m}),$$
for $a_1, a_2, \ldots, a_n, b_1, b_2, \ldots, b_m$ words of rank 1 and $v_1, v_2, \ldots, v_n, u_1, u_2,\\ \ldots, u_n,
w_1, w_2, \ldots, w_m, x_1, x_2, \ldots, x_m$ vertices. Then 
\begin{align*}
	a \circ b &=  (\iota_{v_1}\otimes a_1 \otimes \iota_{u_1})\circ (\iota_{v_2}\otimes a_2 \otimes \iota_{u_2}) \circ \ldots \circ (\iota_{v_n}\otimes a_n \otimes \iota_{u_n})\\
	&\circ (\iota_{w_1}\otimes b_1 \otimes \iota_{x_1})\circ (\iota_{w_2}\otimes b_2 \otimes \iota_{x_2}) \circ \ldots \circ (\iota_{w_m}\otimes b_m \otimes \iota_{x_m}),
\end{align*}
which is of the desired form.

In the case that $w = a \otimes b$, we have the following equalities:
 $$a \otimes b =(\iota_{\mathbf{t}(a)} \circ a) \otimes (b \circ \iota_{\mathbf{s}(b)}) =  (\iota_{\mathbf{t}(a)} \otimes b) \circ (a \otimes \iota_{\mathbf{s}(b)}).$$

Applying the inductive hypothesis, we then have
\begin{align*}a \otimes \iota_{\mathbf{s}(b)} &= \left( (\iota_{v_1}\otimes a_1 \otimes \iota_{u_1}) \circ \ldots \circ (\iota_{v_n} \otimes a_n \otimes \iota_{u_n}) \right) \otimes \iota_{\mathbf{s}(b)}\\
&= \left( (\iota_{v_1}\otimes a_1 \otimes \iota_{u_1}) \circ \ldots \circ (\iota_{v_n} \otimes a_n \otimes \iota_{u_n}) \right) \otimes (\iota_{\mathbf{s}(b)}\circ \ldots \circ \iota_{\mathbf{s}(b)})\\
&=  (\iota_{v_1}\otimes a_1 \otimes \iota_{u_1} \otimes \iota_{\mathbf{s}(b)}) \circ \ldots \circ (\iota_{v_n} \otimes a_n \otimes \iota_{u_n} \otimes \iota_{\mathbf{s}(b)})  \\
&=  (\iota_{v_1}\otimes a_1 \otimes \iota_{u_1 \otimes \mathbf{s}(b)}) \circ \ldots \circ (\iota_{v_n} \otimes a_n \otimes \iota_{u_n \otimes \mathbf{s}(b)}),  
\end{align*}
for some words $a_1, \ldots, a_n$ of rank $1$ and some vertices $u_i, v_i$, where $1\leq i \leq n$. Therefore, $a \otimes \iota_{\mathbf{s}(b)}$ is of the desired form, and a similar argument shows that the same is true for $\iota_{\mathbf{t}(a)} \otimes b$. Hence, their composition has the desired form and thus the claim holds to $w = a \otimes b$. By the principle of mathematical induction, it follows that all words in $A(G)$ can be decomposed into such a form. 
\end{proof}

Lemma~\ref{generalform} has a remarkable consequence when it is applied to the tensor categories $\mathbf{OTa}$ and $\Web$. In particular, it will allow us to think of any tangle (or web) as a composition of tangles (or webs) that each looks like the union of a generating tangle (or web) with a finite number of straight vertical lines. 

\begin{defn}\label{freetensor}
Let $G$ be a small directed multi-graph with the vertex set of $G$ having a monoid structure $(V(G), \otimes, \varnothing)$. Then the \textit{free tensor category over the graph G} is denoted by $FG$ and is defined by the following data:
\begin{itemize}
\item \textbf{Objects:} The objects in $FG$ are the vertices of $G$.
\item \textbf{Morphisms ($X \to Y$):} These are equivalence classes of words in the alphabet of $G$.
\item \textbf{Composition:} This operation is obtained by performing the move (A3) move described in Definition~\ref{alphabet} on any two representatives of the equivalence classes of moves.
\item \textbf{Tensor Product:} This operation is obtained by performing the move (A4) described in Definition~\ref{alphabet} on representatives of the equivalence classes of moves.
\item \textbf{Identity ($X$):} $Id_{X} = \overline{\iota_X}$.
\item \textbf{Tensor Unit:} $\varnothing$.
\end{itemize} 
\end{defn}
The term ``free" in this definition is motivated by the following property of this tensor category.

\begin{thm}\label{FreeProperty}
	Let $FG$ be the free tensor category over a graph $G$ with the properties described in Definition ~\ref{freetensor} and with vertex set $V(G)$ and edge set $E(G)$. For any category $\mathcal D \in V(G)-\mathbf{Cat^\otimes}$ and any function $\phi: E(G) \to \Mor(\mathcal D)$, such that for all $a \in E$,
	$$\bold s(\phi(a)) = \bold s(a) \quad \text{ and } \quad \bold t(\phi(a)) = \bold t(a),$$
	we have that there exists a unique $V(G)-\mathbf{Cat^\otimes}$-morphism $\phi': FG \to \mathcal D$ such that $\phi'(\bar a) = \phi(a)$ for all $a \in E(G).$
\end{thm}
\begin{proof}
	We denote the elements of $FG$ by $\bar a$, where $a \in A(G)$ and the bar denotes the equivalence class of words described in Definition~\ref{EquivWords}.
	First we extend $\phi$ to a well defined function $\phi: A(G) \to \D$ inductively. For words of rank 1 we will define
	$$\phi(a) = \begin{cases}
		 \phi(a)& \text{if} \, a \in E(G)\\
		 Id_{\bold s(a)}& \text{if $a$ is an empty path}.\\
	\end{cases}
		$$
	Having defined $\phi$ on words of rank $\leq n$, for any word $c$ of rank $n+1$ then there exists a unique $a, b \in A(G)$ of rank at most $n$ such that $c = a \circ b$ or $c = a \otimes b$. We then define 
	$$\phi( c) = \begin{cases}
		\phi(a)\circ \phi(b)& \text{if} \, c = a \circ b \\
		\phi(a) \otimes \phi(b)& \text{if} \, c = a \otimes b. \\
	\end{cases}$$
Then we define $\phi'(\bar a) = \phi(a)$ for all $\bar a \in \Mor(FG)$.
By the way that $\phi'$ is constructed, we can quickly verify that $\phi'$ is functorial, tensor-product-preserving, and that it acts as the identity on objects. Therefore, all that is needed to show for $f$ to be a morphism of $V(G)-\mathbf{Cat^\otimes}$ is that $\phi'$ is a well defined function.

 By the definition of the equivalence relation in $FG$, it suffices to show that $\phi(w_1) = \phi(w_2)$, for any two words $w_1$ and $w_2$,  where $w_2$ is obtained from $w_1$ by replacing a subword $a_1$ of $w_1$ with another subword $a_2$ using one of the moves (1) through (6) in Definition~\ref{EquivWords}. To see this, note that in all six cases of the relation between $a_1$ and $a_2$, $\phi(a_1)  = \phi(a_2)$ since $D$ is a strict tensor category. For $ b \in \Mor(\D)$, we define the following functions:
\begin{itemize}
	\item[(i)] $\overline{A(3)^L_{b}}: \{c \in \Mor(\D): \mathbf{t}(c) = \mathbf{s}(b)\} \to \Mor(\D)$; $ c \mapsto   b  \circ  c$ \\
	\item[(ii)] $\overline{A(3)^R_b}: \{ c \in \Mor(\D): \mathbf{t}(b) = \mathbf{s}(c)\} \to\Mor(\D)$; $c \mapsto  c \circ  b$ \\
	\item[(iii)] $\overline{A(4)^L_{b}}:\Mor(\D) \to\Mor(\D)$; $ c \mapsto b \otimes c$ \\
	\item[(iv)] $\overline{A(4)^R_{ b}}:  \Mor(\D) \to  \Mor(\D)$; $ c \mapsto  c \otimes  b.$
\end{itemize}
Let $\{f_k\}_{k=1}^n$ be a word decomposition of $w_1$ starting at $a_1$. For each $k \in \{1, \dots, n\}$, there exists $b_k \in A(G)$, such that $$f_k \in \{A(3)_{b_k}^L, A(3)_{b_k}^R, A(4)_{b_k}^L, A(4)_{b_k}^R\}.$$ Define $\bar f_k$ as the corresponding map with $$\bar f_k \in \{\overline{A(3)_{\phi(b_k)}^L}, \overline{A(3)_{\phi(b_k)}^R}, \overline{A(4)_{\phi(b_k)}^L}, \overline{ A(4)_{\phi(b_k)}^R}\}.$$ 
Then by definition of $\phi$ we have that $\phi(f_k(c)) = \bar f_k( \phi(c))$, for all $c \in \Dom( f_k)$. By iterating this, we have 
\begin{align*}
	\phi(w_1)& = \phi(f_n \circ f_{n-1} \circ \cdots \circ f_1(a_1))\\ 
	&= \bar f_n \circ \bar f_{n-1} \circ \cdots \circ \bar f_1(\phi(a_1))\\
	& = \bar f_n \circ \bar f_{n-1} \circ \cdots \circ \bar f_1(\phi(a_2))\\
	& = \phi(f_n \circ f_{n-1} \circ \cdots \circ f_1(a_2)) = \phi(w_2).
\end{align*}
Therefore, $\phi'$ is well defined.


We show uniqueness by using induction on rank once again. Suppose $\phi'$ and $\phi''$ are morphisms of $V(G) -\mathbf{Cat^\otimes}$ that extend $\phi$. By our assumption and the property that functors preserve identity maps, we have that $\phi'$ and $\phi''$ must agree for all  words of rank 1. Additionally, if $\phi'$ and $\phi''$ agree for all words of rank at most n, then if $c$ is a word of rank $n+1$, there exist $a, b \in A(G)$ of rank at most $n$ such that $c = a \circ b$ or $c = a \otimes b$. In either case we have 
$$\phi''(\bar c) = \phi''(\bar a)\circ \phi''(\bar b) = \phi'(\bar a)\circ \phi'(\bar b) = \phi'(\bar c)$$
or 
$$\phi''(\bar c) = \phi''(\bar a)\otimes \phi''(\bar b) = \phi'(\bar a)\otimes \phi'(\bar b) = \phi'(\bar c).$$
Thus, for all words $c$ in $A(G)$, we have that $\phi''(\bar c) = \phi'(\bar c)$ and thus $\phi' = \phi''$.
\end{proof}
\begin{defn} \label{presentation}
	Let $(\Ca, \otimes, \varnothing)$ be a small tensor category and let $M \subseteq Mor(\Ca)$. We can form a directed graph, $G_M$, whose edges are the morphisms in $M$ and whose vertices are the elements of the monoid $(Ob(\Ca), \otimes, \varnothing)$. Let $R$ be a subset of $Mor(FG_M) \times Mor(FG_M)$, such that if $(f, g) \in R$ then $f$ and $g$ have the same source and target. Then we say that $\langle M : R\rangle$ is a \textit{presentation of $\Ca$ with generators $M$ and relations $R$} if and only if there exists an isomorphism 
	$$\phi: FG_M/\langle R\rangle \to  \Ca$$
	in the category $Ob(\Ca) - \mathbf{Cat}^\otimes$, such that $\phi(f) = f$ for all $f \in M$.
\end{defn}

\begin{thm}\label{genrel}
	Suppose that $(\Ca, \otimes, \varnothing)$ is a small tensor category presented by $\langle M : R \rangle$. Then for any category $\mathcal D  \in Ob(\mathcal C)-\mathbf{Cat}^\otimes$ and any function $\phi: M \to Mor(\mathcal D)$ such that $(f, g) \in R$ implies that $\phi(f) = \phi(g)$, we have that there exists a unique morphism $\phi'$ in $Ob(\mathcal C)-\mathbf{Cat}^\otimes$ such that $\phi'(f) = \phi(f)$, for all $f \in M$.
\end{thm}
\begin{proof}
	By Theorem~\ref{FreeProperty}, we have that there exists a $\psi$ in $Ob(\mathcal C)-\mathbf{Cat}^\otimes$, such that $\psi: FG_M \to \mathcal D$ and $\psi(f) = \phi(f)$, for all $f \in M$. Now let $(f, g) \in R$. This implies that $\psi(f)= \phi(f) = \phi(g)  = \psi(g),$ so we have $(f, g) \in ker(\psi)$. Thus, since $ker(\psi)$ is a tensor congruence that contains $R$, by the definition of the tensor congruence generated by $R$, we have that $\langle R \rangle \subset  ker(\psi).$ To this end, we consider the map 
	$$\pi: FG_M/\langle R\rangle \to FG_M/ker(\psi),$$
	where for each morphism $f \in FG_M$, $\pi$ sends the class of $f$ in $FG_M/\langle R\rangle$ to the class of $f$ in $FG_M/ker(\psi)$ (this is a well defined map since $\langle R \rangle \subset  ker(\psi)$). Now, since $\mathcal C$ is presented by $\langle M:R \rangle$,  we know that there exists a morphism in $Ob(\mathcal C)-\mathbf{Cat}^\otimes$,
	$$\alpha: \mathcal C \to FG_M/\langle R\rangle,$$
	such that $\alpha$ sends every $f \in M$ to the class of $f$ in $FG_M/\langle R\rangle$. Finally, by Theorem ~\ref{fit}, there exists a morphism 
	$$\overline{\psi}: FG_M/ker(\psi) \to \mathcal D.$$
	Moreover, the proof of this theorem lets us choose this morphism $\overline{\psi}$ such that for all $f \in FG_M$, $\overline{\psi}([f]) =\psi(f).$
	
	Choose $\phi' = \overline{\psi}\circ \pi \circ\alpha.$ Then indeed $\phi'$ is a morphism of $Ob(\mathcal C)-\mathbf{Cat}^\otimes$ and for all $f \in M$, we have 
	$$\phi'(f) = \overline{\psi}\circ \pi \circ\alpha(f) = \overline{\psi}([f]) = \psi(f) = \phi(f).$$
	
	Suppose that there existed two extensions of $\phi$ to $\mathcal{C}$, and call them $\phi'$ and $\phi''$. Then let $\alpha$  be the isomorphism $\mathcal C \to FG_M/\langle R\rangle$ that sends each $f \in M$ to the class of $f$ in $FG_M/\langle R\rangle$, which we will denote $\overline{f}$. Additionally let $\pi$ be the morphism in $Ob(\mathcal C)-\mathbf{Cat}^\otimes$ given by
	$$ FG_M \to  FG_M/\langle R\rangle, \,\,\, \overline{f} \mapsto [f],$$
	where $[f]$ is the class of $f$ in  $FG_M/\langle R\rangle$.
	Then we have that 
	$\phi'' \circ \alpha^{-1} \circ \pi$ and $\phi' \circ \alpha^{-1} \circ \pi$ are morphisms on the free category $FG_M$ that agree on the elements of $M$. By Theorem~\ref{FreeProperty}, we can conclude that
		$$\phi'' \circ \alpha^{-1} \circ \pi  = \phi' \circ \alpha^{-1} \circ \pi.$$ 
	Therefore, for any morphism $[g] \in  FG_M/\langle R\rangle$, $\pi(\overline{g}) = [g]$, and thus 
	$$\phi''\circ \alpha^{-1}([g]) = \phi''\circ \alpha^{-1}\circ \pi(\overline{g})  = \phi'\circ \alpha^{-1}\circ \pi(\overline{g})  =  \phi'\circ \alpha^{-1}([g]). $$
	Hence, $\phi''\circ \alpha^{-1} = \phi'\circ \alpha^{-1}$ and consequently $\phi' = \phi''.$
\end{proof}

\section{The category of oriented webs} \label{webdefine}

In this section we introduce the category of $sl(3)$ webs, which are the objects we are primarily working with in this paper.
\begin{defn} \label{webs}
	Let $\epsilon = (\epsilon_1,\ldots , \epsilon_m)$ and $\nu = (\nu_1, \ldots, \nu_n)$ be finite strings of the symbols $+$ and $-$ (we also allow the empty string).  We say an \textit{$(\epsilon, \nu)$-web} is an oriented planar multigraph $W$, possibly with finite number of verticeless loops, in the strip $S:= \R \times [0,1]$. We denote the set of vertices of $W$ by $V(W)$ and we require that a $(v_1, v_2)$-edge is a $C^1$ orientation-preserving embedding $\phi : [0,1] \to S$ such that $\phi(0) = v_1$ and $\phi(1) = v_2$. We impose that the following hold:
	\begin{enumerate}
		
		\item[(i)] For all vertices $v \in S^{\circ}$, $\text{deg}(v) = 3$ and $v$ is either a sink or a source. In such a case, we say $v$ is an \textit{interior vertex} of $W$. 
		\item[(ii)] $ V(W)\cap (\R \times \{0\}) = \{(i,0) : i \in \{1, 2, \ldots, m\}\}$ and all such vertices have degree 1
		such that $(i,0)$ is a source if $\epsilon_i = +$ and a sink if $\epsilon_i = -$.
		\item[(iii)] $ V(W)\cap (\R \times \{1\}) = \{(j,1) : j \in \{1, 2, \ldots, n\}\}$ and all such vertices have degree 1
		such that $(j,1)$ is a source if $\nu_j = -$ and a sink if $\nu_j = +$.
		\item[(iv) ]Edges of $W$ can only intersect the boundary of the strip at its manifold boundary. Moreover, at these intersection points there must exist a neighborhood around the boundary point in which the edge coincides with a line perpendicular to $\R \times \{0\}$.
	\end{enumerate}
\end{defn} 
In Fig.~\ref{webex}, we show an example of a web with source $\epsilon = (+, +, -, -, +)$ and target $\nu = (-, +, +)$. We note that an interior vertex is a \textit{sink} or a \textit{source} if the edges meeting at the vertex are oriented towards the vertex or, respectively, away from the vertex.

\begin{figure}[h]
	\includegraphics[scale=0.25]{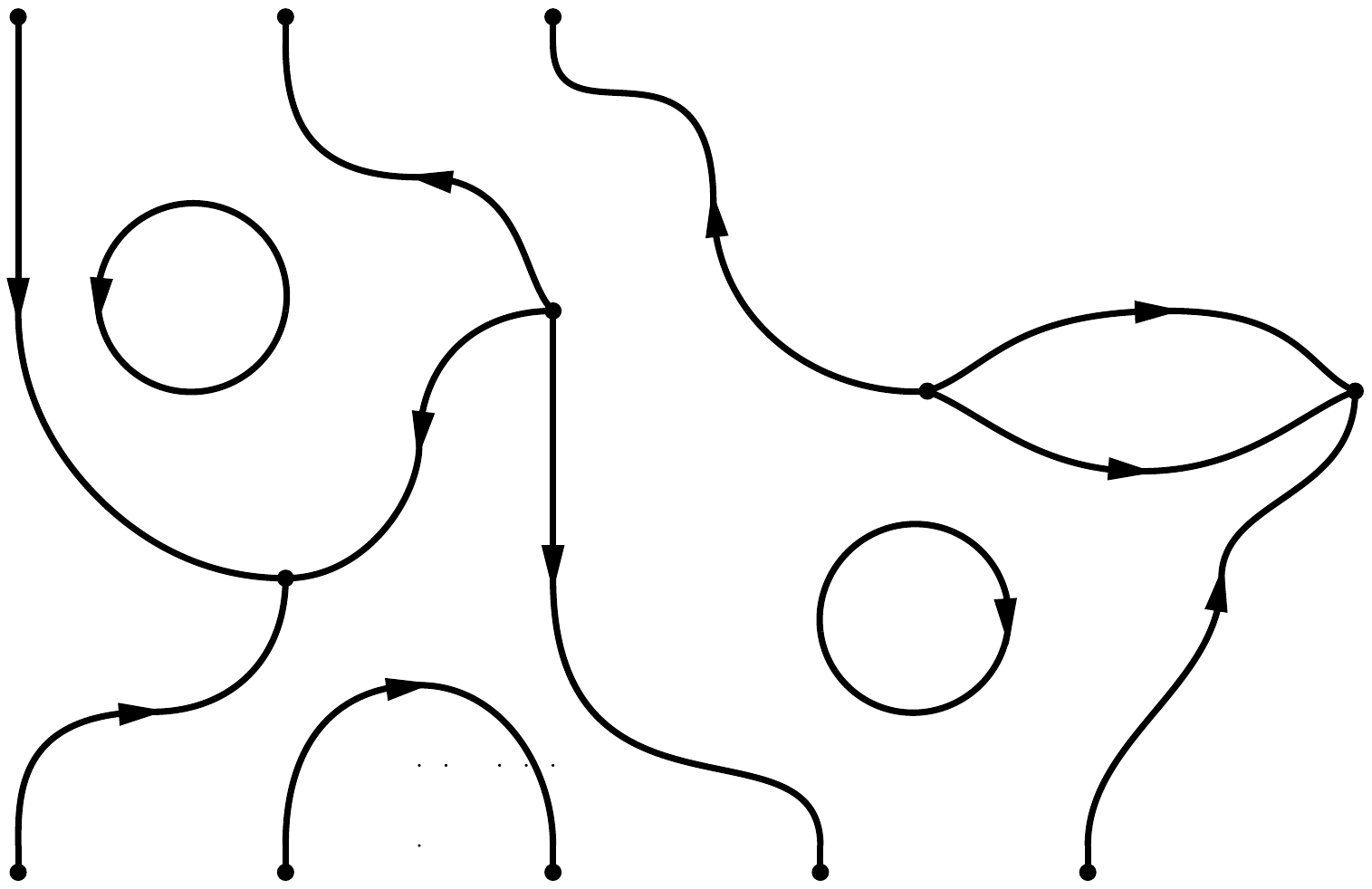}
	\put(-185,125){\fontsize{10}{10}$-$}
	\put(-148,125){\fontsize{10}{10}$+$}
	\put(-113,125){\fontsize{10}{10}$+$}
	\put(-185,-10){\fontsize{10}{10}$+$}
	\put(-148,-10){\fontsize{10}{10}$+$}
	\put(-113,-10){\fontsize{10}{10}$-$}
	\put(-77,-10){\fontsize{10}{10}$-$}
	\put(-40,-10){\fontsize{10}{10}$+$}
	\caption{A $((+,+,-,-, +), (-,+,+))$-web}
	\label{webex}
\end{figure}

By taking the union over all vertices and the images of edges, we can view every web as a subset of the strip $S$.

Given any two webs $W_1$ and $W_2$, we say that they are \textit{isotopic} if and only if there is an isotopy of $S$ relative to $\partial S$ that sends $W_1$ to $W_2$. That is, there exists a continuous map

$$H: S\times [0,1] \to S,$$
such that each $H(\cdot, t)$ is a homeomorphism on $S$ that gives the identity when restricted to the boundary of $S$, $H(\cdot, 0) = Id_S$, and $H(W_1, 1) = W_2$. With some work, one can check that isotopy is an equivalence relation, allowing us to identify every web $W$ with its isotopy class $[W]$. 

These webs then naturally form a category, with composition being concatenation. The definition of this category is as follows.

\begin{defn} \label{websCat}
	The \textit{category of webs}, denoted here by $\mathbf{Web}$, is a small category consisting of the following data: 
	
	\begin{itemize}
		\item \textbf{Objects:} Finite and possibly empty strings of the symbols $+$ and $-$.
		\item \textbf{Morphisms ($\epsilon \to \nu$):} Isotopy classes of $(\epsilon, \nu)$-webs.
		\item \textbf{Composition:} Given morphisms $[W]:  \epsilon \to \nu$ and $[V]: \nu  \to \mu$, $[V] \circ [W]$ is given by first taking any web from each equivalence class, say $W$ and $V$, respectively. Then, we construct an $(\epsilon, \mu)$-web, $V \circ W$, by geometrically stacking $V$ on top of $W$, deleting the vertices in $\R \times \{1\}$, and linearly scaling the $y$ axis in half so that $V \circ W$ lies in $S$. Then we define $[V] \circ [W]$ as the isotopy class of $V \circ W$. That is, $[V \circ W] = [V] \circ [W]$.
		\item \textbf{Identity ($\epsilon$):} This is the isotopy class of the $(\epsilon, \epsilon)$-web with only edges between points of the form $(i, 0)$ and $(i,1)$, for $i \in \mathbb{N}$, that are straight lines with the appropriate orientation. 
	\end{itemize}
	\end{defn}

We can further enrich this category by defining a bifunctor $\otimes: \mathbf{Web} \times \mathbf{Web} \to \mathbf{Web}$. On objects we define 
\[ 
(\epsilon_1, \ldots, \epsilon_n) \otimes (\nu_1, \ldots, \nu_m) = (\epsilon_1, \ldots, \epsilon_n,\nu_1, \ldots, \nu_m).
\]

Given webs $W$ and $V$, the tensor product $[W] \otimes [V]$ is the isotopy class of the web obtained by placing an isotopic copy of $V$ to the right of $W$ so that the two are disjoint and the resulting web satisfies properties (ii) - (iv) in Definition~\ref{webs} (see Fig.~\ref{Tangle Tensor}). This tensor product makes the triple $(\mathbf{Web}, \otimes, \varnothing)$ a strict tensor category.

\begin{figure}[h]
	\includegraphics[scale=0.4]{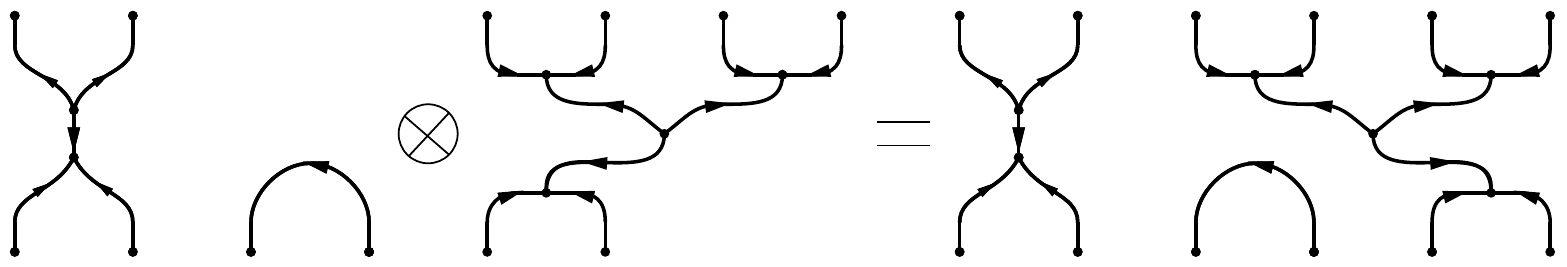}
	\caption{A tensor product of webs}
	\label{Tangle Tensor}
\end{figure}

\section{A Presentation for the category $\mathbf{Web}$} \label{webpres}

In this section we provide a presentation for the category $\mathbf{Web}$ using generators and relations, by means of Definition~\ref{presentation}.
For this purpose, we start by giving names to some special isotopy classes of webs and their represented diagrams shown in  Fig.~\ref{SpecialWebs}. 

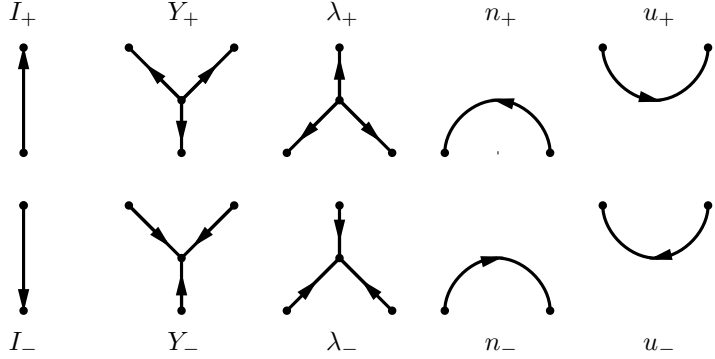
\begin{figure}[ht]
	\includegraphics[scale=0.7]{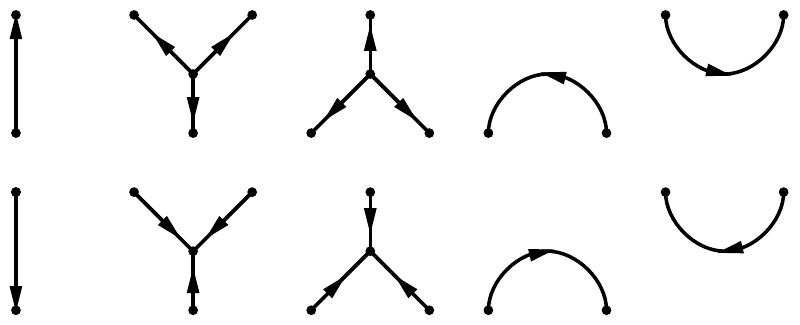}
		 \put(-270,115){\fontsize{10}{10}$I_+$}
		 \put(-210,115){\fontsize{10}{10}$Y_+$}
		 \put(-150,115){\fontsize{10}{10}$\lambda_+$}
		 \put(-90,115){\fontsize{10}{10}$n_+$}
		 \put(-30,115){\fontsize{10}{10}$u_+$}
		  \put(-270,-10){\fontsize{10}{10}$I_-$}
		 \put(-210, -10){\fontsize{10}{10}$Y_-$}
		 \put(-150,-10){\fontsize{10}{10}$\lambda_-$}
		 \put(-90,-10){\fontsize{10}{10}$n_-$}
		 \put(-30,-10){\fontsize{10}{10}$u_-$}
	\caption{In the order listed above, we will refer to these webs as: $I_+, Y_+, \lambda_+, n_+, u_+, I_-, Y_-, \lambda_-, n_-, u_-$.}
	\label{SpecialWebs}
\end{figure}

 It is not hard to see that the following relations hold:
 \begin{align}
 \text{For } \epsilon = (\epsilon_1, \ldots, \epsilon_n); \: \bigotimes_{i=1}^{n}I_{\epsilon_i} &= Id_\epsilon, \text{where} \,\, \epsilon_i \in \{ + ,- \}, \label{RI}\\ 
 (n_\mp \otimes I_\pm) \circ (I_{\pm} \otimes u_\pm) &= I_\pm =   (I_\pm \otimes n_\pm)\circ( u_{\mp} \otimes I_{\pm}), \label{R1}\\ 
 (n_\pm \otimes I_\pm) \circ(I_\mp \otimes Y_\pm) &= \lambda_\pm =  (I_\pm \otimes  n_\mp)\circ(Y_\pm \otimes I_\mp).  \label{R2}
 \end{align}
 The first of the above equalities comes from the definition of the identity element in $\Web$. The second and third double equalities hold due to the planar isotopies depicted in Fig.~\ref{relations2-3}.

  These planar isotopies of webs hold for all possible orientations. In fact,  these relations are all that is needed to describe the category of webs
 \begin{figure}[h]
		 $$ \raisebox{-15pt}{\includegraphics[height=1.5cm]{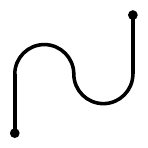}} \mathlarger{\mathlarger{\leftrightarrow}} \raisebox{-15pt}{\includegraphics[height=1.5cm]{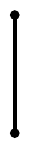}}\mathlarger{\mathlarger{\leftrightarrow}} \raisebox{-15pt}{\includegraphics[height=1.5cm]{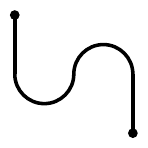}}
		 \hspace{0.5cm} \text{;} \hspace{0.5cm}
		 \raisebox{-15pt}{\includegraphics[height=1.5cm]{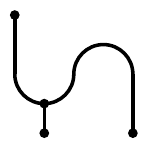}} \mathlarger{\mathlarger{\leftrightarrow}} \raisebox{-15pt}{\includegraphics[height=1.5cm]{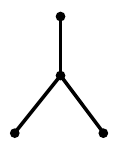}}\mathlarger{\mathlarger{\leftrightarrow}} \raisebox{-15pt}{\includegraphics[height=1.5cm]{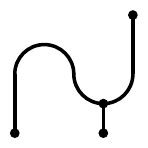}}$$
         \caption{Planar isotopies of webs}
	\label{relations2-3}
\end{figure}

To formalize this claim, we will utilize the definition of a tensor category presentation given in Definition~\ref{presentation}.

\begin{thm} The morphisms, $ Y_+, n_+, u_+, Y_-, n_-, u_-$ in $\Web$, depicted in Fig.~\ref{SpecialWebs}, together with the relations 

\begin{itemize}
\item[R1a:] $(n_\mp \otimes I_\pm) \circ (I_{\pm} \otimes u_\pm) \sim I_\pm$
\item[R1b:] $I_\pm \sim (I_\pm \otimes n_\pm)\circ( u_{\mp} \otimes I_{\pm})$
\item[R2:] $(I_\pm \otimes  n_\mp)\circ(Y_\pm \otimes I_\mp) \sim (n_\pm \otimes I_\pm) \circ(I_\mp \otimes Y_\pm) $
\end{itemize}
form a presentation of the category $\Web$. That is to say that
$$\langle  Y_+, Y_-, n_+, n_{-}, u_{+}, u_{-} : R1a, R1b, R2\rangle \,\,\, \text{presents} \,\,\, \Web.$$
\end{thm}

\begin{proof}

Consider the abstract directed graph $G$ whose vertices are the objects of $\Web$ and whose edges are the isotopy classes of the set $K:= \{Y_+, Y_-, n_+, n_-, u_+, u_-\}$. We must show that the quotient category $FG/\langle R1a, R1b, R2 \rangle$, where $FG$ is the free tensor category generated by $G$, is isomorphic to $\Web$ in the category $Ob(\Web) - \mathbf{Cat}^\otimes$. Thus, by Theorem~\ref{fit} and Definition~\ref{presentation}, it suffices to construct a surjective tensor-product-preserving functor $\phi: FG \to \Web$, such that its kernel (see Example~\ref{ker}) is the tensor congruence $\langle R1a, R1b, R2 \rangle.$ 

There is an obvious tensor product preserving functor, $\phi: FG \to \Web$, that is the identity on objects, takes each one of these special webs to itself in the category $\Web$, and takes the identities on $+$ and $-$ to $I_+$ and $I_-$, respectively. 


To prove that the image of this functor is the category $\Web$, consider an arbitrary morphism $[W]$ in $\Web$. Given the projection function $\pi: \R \times [0,1] \to [0,1]$, we can define the following singularities:


\begin{itemize}
\item We say that $x \in W$ is a \textit{type-I singularity} if and only if it is a vertex of $W$.
\item We say that $x \in W$ is a \textit{type-II singularity} if and only if it is a local minimum or maximum of the height function $\pi$.
\end{itemize}

We can see that by definition, $W$ has a finite number of type-I singularities, but not necessarily finitely many type-II singularities. However, we may construct an isotopic copy of $W$, call it $W'$, with the following properties.

\begin{itemize}
\item[(1)] $\pi^{-1}\{x\} \cap W'$ is finite for all $x \in [0,1]$;
\item[(2)] $W'$ has finitely many type-II singularities
\item[(3)] No point of $W'$ is both a type-I and type-II singularity;
\item[(4)]  For all $x \in [0,1]$, the line $\R \times \{x\}$ is never tangent to the web $W'$, except when the intersection point is a type-II singularity;
\item[(5)] For all $x \in [0,1]$, $\pi^{-1}\{x\} \cap W'$ contains no more than one singularity.
\end{itemize}

If these conditions hold for some web $W'$, we say that $W'$ is in \textit{general position}. If the existence of such an isotopic copy $W'$ is not clear, one can always first take $W$ via an ambient isotopy to a web whose edges are piecewise-linear (see \cite{Crom}, Theorem 1.11.6), then by perturbing the vertices (possibly including the joints on each edge) we may take this piecewise-linear web to a piecewise-linear web $W''$ that is in general position. Then we can use an ambient isotopy that takes a neighborhood of $W''$ around each joint of each edge to some arc of a circle in such a way that the resulting web $W'$ has $C^1$ edges (see \cite{Crom}, Theorem 1.11.7). Since all this isotopy did was smooth out the joints with circle arc, $W'$ will still satisfy conditions $(1)-(4)$ of a general position webs. With one more isotopy locally perturbing the singularities of $W'$ we get a web in general position.

Denote the number of singularities on $W'$ with the natural number $n$. Given such a web $W'$ in general position, we can consider $n+1$ horizontal parallel lines slicing $W'$, so that between any two consecutive lines there is exactly one singularity and such that the first and last lines, from bottom to top, are $\R \times \{0\}$ and $\R \times \{1\}$, respectively. 
This will result in $n$ webs  $W'_1, \ldots, W'_n$, where $W'_i$ is the section between the $i$th and $(i+1)$th lines and $W'_n \circ \ldots \circ W'_1 = W'$, and where each $W'_i$ ($1 \leq i 
\leq n$) satisfies the definition of a web.


 Each of these webs $W'_i$ is then the union of a finite number of disjoint connected components, where only one of these components contains a singularity. The components containing no singularities are isotopic to $I_{\pm}$ and the  components with a singularity, are isotopic to either $\lambda_{\pm}$ or $Y_{\pm}$ if they contain a singularity of type-I or to $n_\pm$ or $u_\pm$ if they contain a singularity of type-II. 
Now, we can use relation $\eqref{R2}$ to rewrite each $\lambda_+$ and $\lambda_-$ in terms of $I_+, I_-, Y_+, n_+, u_+, Y_-, n_-$, and $u_-$. Indeed, $\lambda_+ = (I_+ \otimes  n_-)\circ(Y_+ \otimes I_-)$ and $\lambda_- = (I_- \otimes  n_+)\circ(Y_- \otimes I_+)$, or 
$\lambda_+ = (n_+ \otimes  I_+)\circ(I_- \otimes Y_+)$ and $\lambda_- = (n_- \otimes  I_-)\circ(I_+ \otimes Y_-)$, as shown in Fig.~\ref{relations2-3}.

Then, we have that each $W'_i$ is isotopic to a web of the form $I_{u_i}\otimes  \phi(X_i) \otimes I_{v_i}$, for some objects $u_i$ and $v_i$ and some $X_i \in K \cup \{\lambda_+, \lambda_- \}$. 

Therefore, we have the equalities of the following isotopy classes: 
\begin{align*}
[W] = [W'] &= (I_{u_n}\otimes \phi(X_n) \otimes I_{v_n})\circ \ldots \circ (I_{u_1}\otimes \phi( X_1) \otimes I_{v_1})
\\ &=\phi((\iota_{u_{n}}\otimes  X_n  \otimes \iota_{v_n})\circ \ldots \circ (\iota_{u_1}\otimes  X_1  \otimes \iota_{v_1})),
\end{align*}
where each $X_i (1 \leq i \leq n)$ is a word in the alphabet of $G$. It follows that the image of the functor $\phi$ is $\Web$. We note that given any web in general position $W$, this construction of a word by decomposing the web in horizontal strips is unique; we will denote this word $d(W)$.

Now we aim to show that the kernel of this functor is the tensor congruence generated by the relations, $R1a$, $R1b$ and $R2$. Since relations $\eqref{R1}$ and $\eqref{R2}$ hold in $\Web$, we must have $\langle R1a, R1b, R2 \rangle \subseteq ker(\phi)$. To this end, let $(a, b) \in ker(\phi)$. Thus $\phi(a) = \phi(b)$. By Lemma~\ref{generalform}, we may then expand $a$ and $b$ as follows
\[a = (\iota_{v_n}\otimes a_n \otimes \iota_{u_n})\circ \ldots \circ (\iota_{v_1}\otimes a_1 \otimes \iota_{u_1})\]
and
\[b =(\iota_{v'_m}\otimes b_m \otimes \iota_{u'_m})\circ \ldots \circ (\iota_{v'_1}\otimes b_1 \otimes \iota_{u'_1}),\]
where each $u_i,v_i, u'_i,v'_i \in Ob(\Web)$ and each $a_i$ and $b_i$ are either identities or elements of $K$. Reducing these words by removing all of the unnecessary identity morphisms will either yield an equivalent word of the form above, where each $a_i, b_i \in K$, or the entire word will be an identity morphism. Therefore, we may assume that each $a_i$ and $b_i$ above contains a singularity. Then 
\[\phi(a)= (I_{v_n}\otimes \phi(a_n) \otimes I_{u_n})\circ \ldots \circ (I_{v_1}\otimes \phi(a_1) \otimes I_{u_1})\]
and 
\[\phi(b) = (I_{v'_m}\otimes \phi(b_m) \otimes I_{u'_m})\circ \ldots \circ (I_{v'_1}\otimes \phi( b_1) \otimes I_{u'_1}),\]
where the right hand sides of the equations above are representatives in general position of $\phi(a)$ and $\phi(b)$. That is, there are webs $W_a$ and $W_b$ in general position such that $d(W_a) = \phi(a)$ and $d(W_b) = \phi(b)$, respectively.

Let $H: S\times [0,1] \to S$ be an isotopy of $S = \mathbb{R} \times [0,1]$ that takes $W_a$ to $W_b$. Then,  $H(W_a, t)$ is in general position for all $t \in [0,1]$ if and only if the isotopy never creates or removes singularities and it never changes the height order of any two singularities.

Thus, under the assumption that $H(W_a, t)$ is in general position for all $t \in [0,1]$, then we can see that the word $d(W_a, t)$ must be the same for all $t \in [0,1]$. Hence the words

\[d(W_a, 0) = (\iota_{v_n}\otimes a_n \otimes \iota_{u_n})\circ \ldots \circ (\iota_{v_1}\otimes a_1 \otimes \iota_{u_1})\]
and
\[d(W_a, 1)=(\iota_{v'_m}\otimes b_m \otimes \iota_{u'_m})\circ \ldots \circ (\iota_{v'_1}\otimes b_1 \otimes \iota_{u'_1})\]
must be the same, resulting in $a$ and $b$ being equivalent words. 

In the case when $H(W_a, t)$ is not in general position for finitely many $t \in [0, 1]$, then the singularities in the height function must have been adjusted. Such a change can be attributed to a finite sequence of the local moves in Fig.~\ref{locrel} along with exchanges of distant critical points and the isotopy that keeps the diagram in general position (for details, we refer the reader to the work by Carter in~\cite[Section 3]{Cart}).

 \begin{figure}[ht] 
	\raisebox{-15pt}{ \includegraphics[scale=0.3]{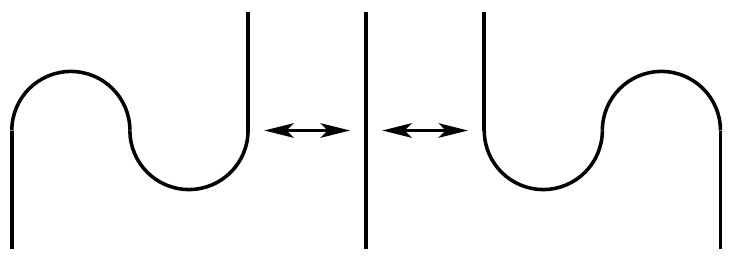}}
	\put(-125, 0){\fontsize{10}{10}$1)$}
	\put(15, 0){\fontsize{10}{10}$2)$}
	\hspace{1cm}
	\raisebox{-15pt}{\includegraphics[scale=0.3]{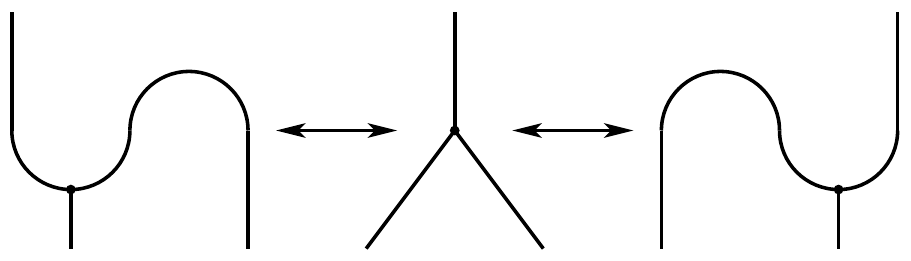}}\\

	\raisebox{-15pt}{
		\includegraphics[scale=0.3]{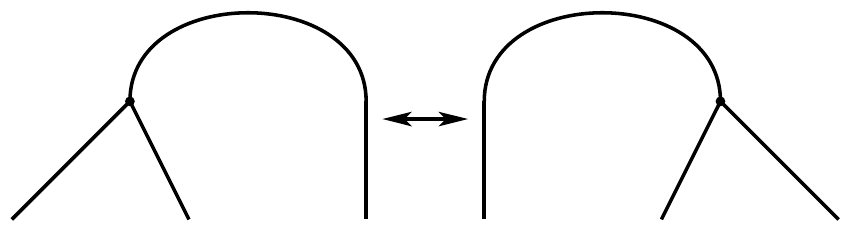}}
	\put(-130, 0){\fontsize{10}{10}$3)$}
	\caption{The three local moves that are sufficient to adjust the singularities of the height function.}
	\label{locrel}
\end{figure}

 Moreover, this sequence can be further expanded to only need the first two of the moves in Fig.~\ref{locrel} if we replace the third move with the finite sequence of moves in Fig.~\ref{move-reduct}.
 \begin{figure}[ht]
	\[ \includegraphics[scale=0.4]{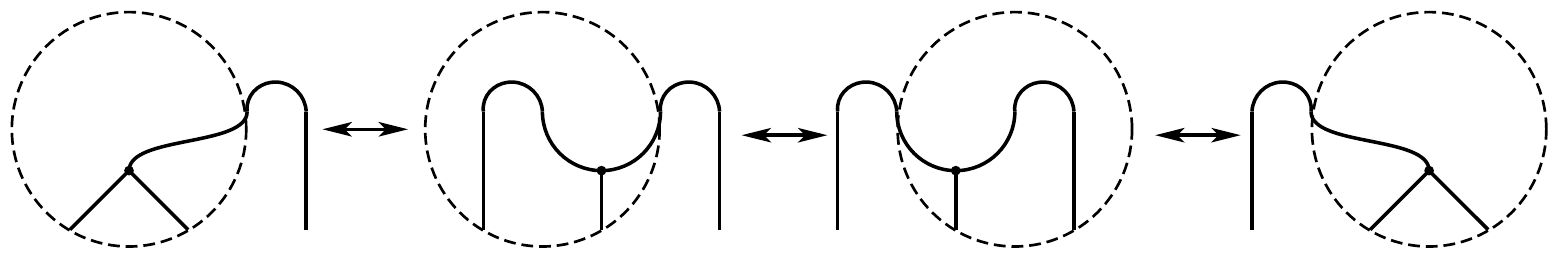}\]
	\label{move-reduct}
		\caption{Move three of Fig.\ref{locrel} is a consequence of move two.}
 \end{figure}



Let $W_1$ and $W_2$ be webs in general position that only differ by the first local move. Then $d(W_2)$ can be obtained from $d(W_1)$ by replacing a subword of the form $(n_\mp \otimes I_\pm) \circ (I_{\pm} \otimes u_\pm)$, $I_\pm$, or $(I_\pm \otimes n_\pm)\circ( u_{\mp} \otimes I_{\pm})$ with another word from that same list. 
However, since these two subwords are related by $R1a$ or $R1b$, then, by the properties of a tensor congruence, the words $d(W_1)$ and $d(W_2)$ must be related by $\langle R1a, R1b, R2 \rangle$. Similarly, since the second move corresponds to a replacement of subwords related by $R2$ move, any two words that only differ by that move must also be related by $\langle R1a, R1b, R2 \rangle$. Finally, any exchange of distant critical points results in equivalent words, due to the equation
 $$(w_1 \otimes \iota_{\mathbf{t}(w_2)})\circ (\iota_{\mathbf{s}(w_1)} \otimes w_2)= w_1 \otimes w_2 = (\iota_{\mathbf{t}(w_1)} \otimes w_2) \circ (w_1 \otimes \iota_{\mathbf{s}(w_2)} ),$$
 which holds for all words $w_1$ and $w_2$. Therefore, $a$ and $b$ must be related by a finite sequence of the relations $\{R1a, R1b, R2 \}$. It follows that $ker(\phi) = \langle R1a, R1b, R2 \rangle$ and thus there is an  isomorphism $FG/\langle R1a, R1b, R2 \rangle \to \Web$. 
\end{proof}


\section{A tangle invariant motivated by the $sl(3)$-polynomial for links}
We start this section by reviewing the definition of the $sl(3)$-polynomial for oriented knots and links. We use a combinatorial construction of this polynomial given by Kuperberg~\cite{Kup}, with a normalization of this invariant which is summarized by Khovanov in~\cite[Section 2]{Khov}. This construction utilizes a one-variable Laurent polynomial, called the Kuperberg bracket, which is evaluated on $sl(3)$ webs using local relations. This, along with properties of presentations of tensor categories allows us to create a functor from the category of oriented tangles, $\mathbf{OTa}$. The target category of this functor is a modified version of $\Web$, $\textbf{LWeb}$, with the $(\epsilon, \nu)$-  morphisms being elements from the $\Z[q, q^{-1}]$-module generated by the $(\epsilon, \nu)$-webs modulo relations that mimic the relations of the Kuperberg bracket. When restricted to links, this functor provides an alternative construction for the $sl(3)$-polynomial.

\subsection{The $sl(3)$-polynomial for knots and links}
\label{sl3link}

The $sl(n)$-polynomial is an invariant for oriented knots and links which is a one-variable specialization of the HOMFLY-PT polynomial invariant (see~\cite{HOMFLY, PT}). In this paper we are concerned only with the $sl(3)$-polynomial and its construction via $sl(3)$ webs and evaluations of these webs, as seen in~\cite{Kup, Khov}.

\begin{defn}
The \textit{Kuperberg bracket} is an isotopy-invariant function $\langle \cdot \rangle$ from the set of closed $sl(3)$ webs, i.e. $(\varnothing, \varnothing)$-webs, to the ring $\Z[q, q^{-1}]$, that is uniquely defined by the following local relations:
\begin{enumerate}
    \item   $\raisebox{-10pt}{\includegraphics[height=0.9cm]{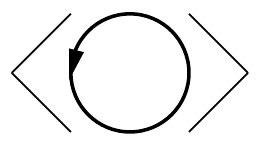}}= q^2 + 1 + q^{-2}$ and $
    \langle \Gamma \cup \raisebox{-7pt}{\includegraphics[height=0.7cm]{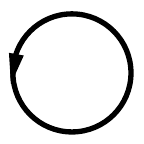}} \rangle = (q^2 + 1 + q^{-2}) \langle \raisebox{-7pt}{\includegraphics[height=0.7cm]{unknot.pdf}} \rangle$,
    \item  
           $ \raisebox{-10pt}{\includegraphics[height=0.9cm]{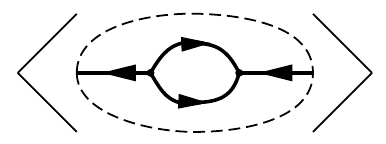}}  = (q + q^{-1}) \raisebox{-10pt}{\includegraphics[height=0.9cm]{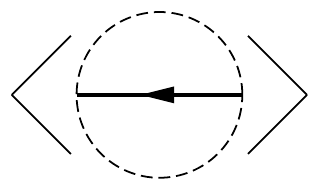}}$,
     \item $
           \raisebox{-15pt}{\includegraphics[height=1.3cm]{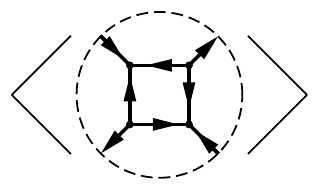}}  =  \raisebox{-15pt}{\includegraphics[height=1.3cm]{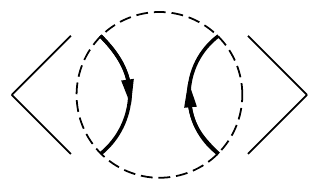}} + \raisebox{-15pt}{\includegraphics[height=1.3cm]{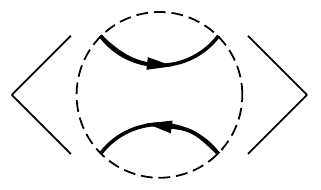}},
        $
\end{enumerate}
where $ \Gamma \cup \raisebox{-7pt}{\includegraphics[height=0.7cm]{unknot.pdf}}$ is the disjoint union between an arbitrary web $\Gamma$ and the standard diagram of the oriented unknot. The diagrams in both sides of relations (2) and (3) are identical, except in a small neighborhood where they differ as shown. Such relations are refer to as \textit{skein relations} in the literature.
 \end{defn}
\begin{rem}
The skein relations in the above definition imply that $\langle \Gamma_1 \cup \Gamma_2 \rangle = \langle \Gamma_1 \rangle  \,  \langle \Gamma_2 \rangle $.
The reason why these local relations define a unique polynomial for a given closed $sl(3)$ web $\Gamma$ is because it can be shown, through an application of Euler's formula, that every closed $sl(3)$ web must contain a loop, digon, or square. 
Thus, the relations (2) and (3) above can be iteratively used to write the evaluation of a closed $sl(3)$ web $\Gamma$ as a $\Z[q, q^{-1}]$-linear combination of evaluations of $sl(3)$ webs with fewer vertices, and simplify $\langle \Gamma \rangle$ until it is written as a linear combination of the brackets evaluated on disjoint union of loops. Finally, the application of relation (1) will give the polynomial $\langle \Gamma \rangle$. 
\end{rem}

Given any link diagram $D$ in $\R \times [0,1]$, we can locally resolve each crossing in one of the two ways depicted in Fig.~\ref{resolutions}.

\begin{figure}[ht]
    \centering
    \includegraphics[scale=0.4]{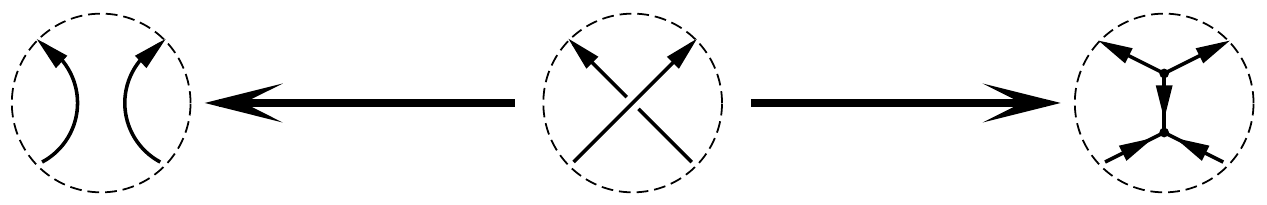}
    \caption{The two choices of resolution for each crossing}
    \label{resolutions}
\end{figure}

By resolving each crossing on $D$ in one of the two possible ways, we obtain an $sl(3)$ web, which we refer to as a \textit{resolution} of $D$. For any oreinted link diagram $D$ with $n$ crossings, there are $2^n$ associated resolutions $\Gamma_i$, where $1\leq i \leq 2^n$. Then, we evaluate $D$ as a $\Z[q,q^{-1}]$-linear combination of evaluations of its resolutions. Specifically,
 \[ p(D):= \sum_{i=1}^{2^n} b_i \langle  \Gamma_i \rangle,\]
where the coefficients $b_i$ are given as the product of the Laurent polynomials obtained by applying the rules in Fig.~\ref{rule} over each resolution $\Gamma_i$ of $D$. 
  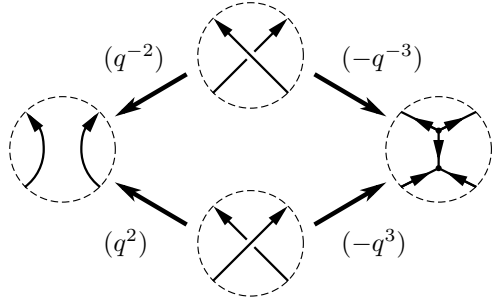
\begin{figure}[ht]
            \centering
            \includegraphics[scale = 0.5]{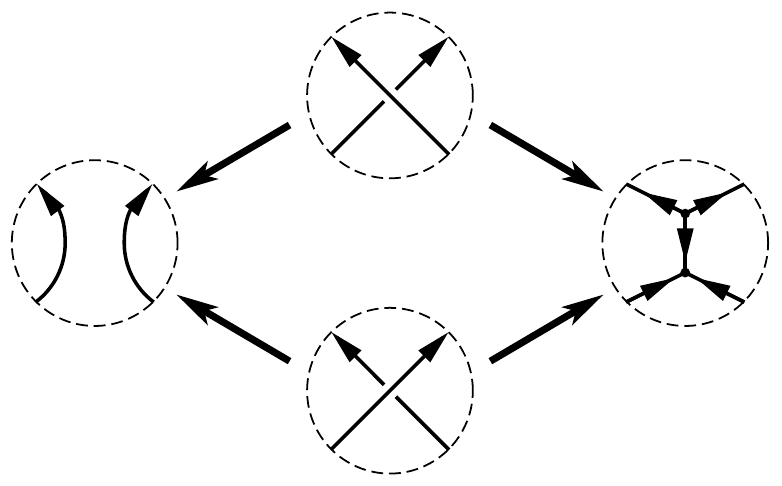}
            \put(-150, 90){\fontsize{10}{10}$(q^{-2})$}
            \put(-150, 20){\fontsize{10}{10}$(q^{2})$}
            \put(-60, 20){\fontsize{10}{10}$(-q^{3})$}
            \put(-60, 90){\fontsize{10}{10}$(-q^{-3})$}
            \caption{The mutiplication rule to get each $b_i$}
            \label{rule}
  \end{figure}

For any oriented knot or link diagrams $D_1$ and $D_2$ that differ by a finite sequence of the Reidemeister moves, it is known that $p(D_1) = p(D_2)$, and thus $p$ is an invariant of oriented links, called the $sl(3)$-polynomial for links. 

\subsection{Constructing the $sl(3)$ tangle invariant} \label{skeinmod}
The scope of this section is to define a functor from the category $\mathbf{OTa}$ of oriented tangles that extends the $sl(3)$-polynomial to oriented tangles.
\begin{defn}
Given sign sequences $\epsilon$ and $\nu$, we define the module $M(\epsilon, \nu)$ to be the free $\Z[q, q^{-1}]$-module generated by isotopy classes of $(\epsilon, \nu)$-webs. We also define the following subsets in $M(\epsilon, \nu)$.
\begin{itemize}
\item We define $N_1(\epsilon, \nu)$ as the set of all elements of the form $$(L \cup W) - (q^2+ 1 + q^{-2})W \in M(\epsilon, \nu),$$ where $L$ an oriented loop disjoint from $W$. 
\item We define $N_2(\epsilon, \nu)$ as the set of all elements of the form $$W_1 - (q+  q^{-1})W_2 \in M(\epsilon, \nu),$$ where $W_2$ can be obtained from $W_1$ by locally replacing a digon with an arc. More formally, there exists words in the alphabet of the set of morphisms $\{u_+, u_-, n_+, n_-, Y_+, Y_-, \lambda_+, \lambda_-\}$, say $a$ and $b$, such that $a = W_1$, $b = W_2$ and $b$ is obtained from $a$ by replacing a subword of the form $\lambda_\pm \circ Y_\mp$ with $I_\pm$. 

\item We define $N_3(\epsilon, \nu)$ as the set of all elements of the form $$W_1 - W_2 - W_3,$$ where $W_2$ is obtained from $W_1$ by locally replacing a square with two of its parallel sides and $W_3$ is obtained from $W_1$ by locally replacing the same square with the other two parallel sides.
More formally, there exists words in the alphabet of the set of morphisms $\{u_+, u_-, n_+, n_-, Y_+, Y_-, \lambda_+, \lambda_-\}$, say $a$, $b$, and $c$, with $a = W_1$, $b = W_2$, and $c= W_3$ such that $b$ is obtained from $a$ by replacing a subword of the form 
$$(\lambda_\mp \otimes \lambda_\pm)\circ (I_\pm \otimes (u_\mp \circ n_\mp) \otimes I_\mp)\circ (Y_\pm \otimes Y_\mp)$$ 
with $I_\mp \otimes I_\pm$, and $c$ is obtained by replacing the same subword of $a$ with $ u_\pm \circ n_\pm$. 
\end{itemize}
\end{defn}

 \begin{figure}[h]
	$$ \raisebox{-20pt}{\includegraphics[height=1.5cm]{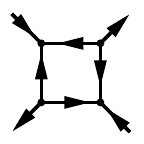}} \mathlarger{\mathlarger{\leftrightarrow}} \raisebox{-20pt}{\includegraphics[height=1.5cm]{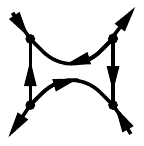}}\mathlarger{\mathlarger{\leftrightarrow}} \raisebox{-20pt}{\includegraphics[height=2.0cm]{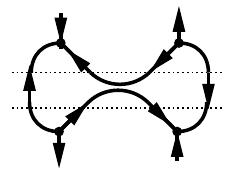}}$$
	\caption{The square and $(\lambda_\mp \otimes \lambda_\pm)\circ (I_\pm \otimes (u_\mp \circ n_\mp) \otimes I_\mp)\circ (Y_\pm \otimes Y_\mp)$ are isotopic.}
	\label{Square isotopy}
\end{figure}

\begin{defn}
Motivated by the formal linear combinations of resolutions in the construction of the $sl(3)$-polynomial invariant for knots and links, we define the \textit{linear web category} $\mathbf{LWeb}$ as the strict monoidal category whose objects are the same of those in $\Web$, but whose morphisms from $\epsilon$ to $\nu$ are elements of the $\Z[q, q^{-1}]$-module $M(\epsilon, \nu)/\langle N_1(\epsilon, \nu), N_2(\epsilon, \nu), N_3(\epsilon, \nu)\rangle$. The composition and tensor product in this category are defined by extending bilinearly the composition and tensor product in $\Web$. That is,
$$(aW_1+bW_2) \circ cW_3 = ac(W_1\circ W_3) + bc(W_2 \circ W_3)$$
and
$$ aW_1 \circ (bW_3 + cW_4)= ab(W_1\circ W_3) + ac(W_2 \circ W_4),$$
for any $a, b, c \in \Z[q, q^{-1}]$ and any composable webs $W_1, W_2, W_3, W_4 \in \Web$.
\end{defn}

Similar to the way that each link diagram $D$ can be associated with a $\Z[q, q^{-1}]$-linear combination of evaluations of closed webs 
$\sum_{i=1}^{2^n} b_i \langle \Gamma_i \rangle$
via the $sl(3)$-polynomial (as shown in Section~\ref{sl3link}),
any oriented $(\epsilon, \nu)$-tangle diagram $T$ can be associated with a morphism from $\epsilon \to \nu$ in $\mathbf{LWeb}$. Even more, this association is a tensor-product-preserving functor from the category of oriented tangles, $\mathbf{OTa}$, to $\mathbf{LWeb}$. However, in order to construct this functor, we will make use of Theorem ~\ref{genrel} and the presentation of $\mathbf{OTa}$
provided by Turaev in~\cite{Tur}.

\begin{thm}~\cite[Theorem 3.2]{Tur}.
\label{tanglerelations}
	The strict tensor category of oriented tangles $\mathbf{OTa}$ is presented by the following elementary  tangles,
	\[
	\includegraphics[scale=0.5]{Pcross2} \quad
	\includegraphics[scale=0.5]{Ncross2} \quad
	\includegraphics[scale=0.5]{u+}    \quad
	\includegraphics[scale=0.5]{u-} \quad
	\includegraphics[scale=0.5]{n+} \quad
	\includegraphics[scale=0.5]{n-} 
	\put(-268, -15){\fontsize{10}{10}$X_+$}
	\put(-218, -15){\fontsize{10}{10}$X_-$}
	\put(-170, -15){\fontsize{10}{10}$u_+$}
	\put(-120, -15){\fontsize{10}{10}$u_-$}
	\put(-70, -15){\fontsize{10}{10}$n_+$}
	\put(-20, -15){\fontsize{10}{10}$n_-$}
	\]
	along with the following relations:
	\begin{itemize}
		\item[$(a)$] $(n_\mp  I_\pm)\circ(I_\pm u_\pm) \sim I_\pm \sim (I_\pm  n_\pm) \circ (u_\mp  I_\pm)$\\
		\item[$(b)$] $( I_-  I_-  n_-) \circ (I_- I_- I_+  n_-  I_-) \circ (I_- I_-  X_\pm  I_-  I_-) \circ ( I_- u_+ I_+ I_- I_-) \circ (u_+ I_- I_-)$\\
		$\sim (n_+I_-I_-) \circ(I_-  n_+I_+I_-I_-) \circ (I_-I_-X_\pm I_-I_-) \circ (I_-I_-I_+u_-I_-) \circ(I_-I_-u_-)$\\
		\item[$(c)$] $X_+ \circ X_- \sim X_- \circ X_+ \sim I_+ I_+$\\
		\item[$(d)$]$(X_+ I_+) \circ (I_+ X_+ ) \circ (X_+ I_+) \sim (I_+X_+) \circ (X_+ I_+) \circ (I_+ X_+)$\\
		\item[$(e)$] $(I_+ n_-) \circ (X_\pm I_-) \circ (I_+ u_-)  \sim I_+$\\
		\item[$(f)$] $(I_+ I_- n_-) \circ (I_- X_+ I_-) \circ (u_+ I_+ I_-)  \circ (n_+ I_+ I_-) \circ (I_- X_- I_-) \circ (I_- I_+ u_-) \sim I_- I_+$\\
		\item[$(g)$] $(n_+ I_+ I_-) \circ (I_- X_- I_-) \circ (I_- I_+ u_-)  \circ  (I_+ I_- n_-) \circ (I_- X_+ I_-) \circ (u_+ I_+ I_-)  \sim I_+ I_-$
	\end{itemize}
	In the above relations, in order to shorten the notation, we use $ab$ to denote the tensor product $a \otimes b$ of two tangles $a$ and $b$.
\end{thm}

\begin{rem}
	Note that these relations correspond to the following isotopies of tangle diagrams:
	
	

		\begin{align*}
			&\text{$a_+$)  } \raisebox{-15pt}{\includegraphics[height=.5in]{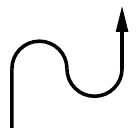}}  \leftrightarrow \raisebox{-15pt}{\includegraphics[height=.5in]{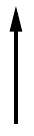}}  \leftrightarrow  \raisebox{-15pt}{\includegraphics[height=.5in]{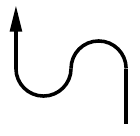}};
			\hspace{2cm} \text{$a_-$)  }\raisebox{-15pt}{\includegraphics[height=.5in]{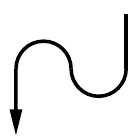}}   \leftrightarrow \raisebox{-15pt}{\includegraphics[height=.5in]{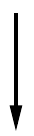}}  \leftrightarrow  \raisebox{-15pt}{\includegraphics[height=.5in]{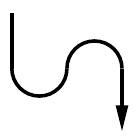}};\\
			&\text{$b$)  } \raisebox{-15pt}{\includegraphics[height=.5in]{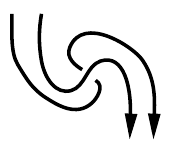}}   \leftrightarrow \raisebox{-15pt}{\includegraphics[height=.5in]{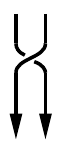}}  \leftrightarrow \raisebox{-15pt}{\includegraphics[height=.5in]{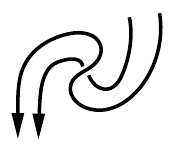}}; \hspace{2cm} \text{$c$)  } \raisebox{-15pt}{\includegraphics[height=.5in]{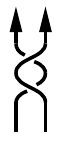}}   \leftrightarrow \raisebox{-15pt}{\includegraphics[height=.5in]{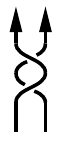}}  \leftrightarrow  \raisebox{-15pt}{\includegraphics[height=.5in]{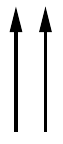}};\\
			&\text{$d$)  } \raisebox{-15pt}{\includegraphics[height=.5in]{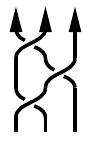}}  \leftrightarrow \raisebox{-15pt}{\includegraphics[height=.5in]{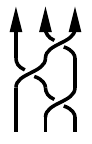}}; \hspace{1.4cm} \text{$e_+$)  } \raisebox{-15pt}{\includegraphics[height=.5in]{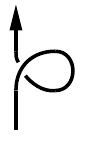}}   \leftrightarrow \raisebox{-15pt}{\includegraphics[height=.5in]{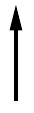}};
			\hspace{1.4cm} \text{$e_-$)  } \raisebox{-15pt}{\includegraphics[height=.5in]{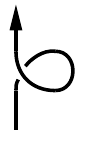}}   \leftrightarrow \raisebox{-15pt}{\includegraphics[height=.5in]{f2.pdf}};\\
			&\text{$f$)  } \raisebox{-15pt}{\includegraphics[height=.5in]{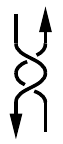}}   \leftrightarrow \raisebox{-15pt}{\includegraphics[height=.5in]{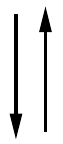}}; \hspace{2cm} \text{$g$)  } \raisebox{-15pt}{\includegraphics[height=.5in]{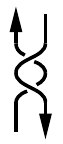}}   \leftrightarrow \raisebox{-15pt}{\includegraphics[height=.5in]{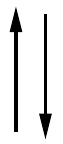}}.
		\end{align*}

	%
	
\end{rem}

\begin{thm} \label{functorF}
The map $F: \mathbf{OTa} \to \mathbf{LWeb}$ sending each object to itself and  sending any oriented $(\epsilon, \nu)$-tangle diagram $T$ to
$$\sum_{i=1}^{2^n} b_i \Gamma_i  + \langle N_1(\epsilon, \nu), N_2(\epsilon, \nu), N_3(\epsilon, \nu)\rangle,$$ where the $\Gamma_i$'s are the enumerated resolutions of tangle $T$ (as $(\epsilon, \nu)$-webs) and each $b_i$ is computed as explained in Fig~\ref{rule}, is a tensor-product-preserving functor.
\end{thm}

\begin{proof}

Since we have presentations for the categories $\mathbf{OTa}$ and $ \mathbf{LWeb}$ we can use Theorem ~\ref{genrel} to define a functor,
$\phi: \mathbf{OTa} \to \mathbf{LWeb},$ by requiring the following assignment on generators 
 \begin{align*}
         \raisebox{-6pt}{\includegraphics[scale = 0.4]{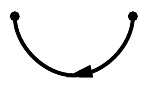}} &\mapsto  \raisebox{-6pt} {\includegraphics[scale = 0.4]{u-.pdf} }&& \raisebox{-6pt}{\includegraphics[scale = 0.4]{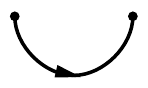}} \mapsto    \raisebox{-6pt}{\includegraphics[scale = 0.4]{u+.pdf}}\\
         \raisebox{-6pt}{\includegraphics[scale = 0.4]{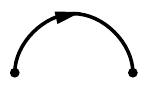}} &\mapsto    \raisebox{-6pt}{\includegraphics[scale = 0.4]{n-.pdf} }&& \raisebox{-6pt}{\includegraphics[scale = 0.4]{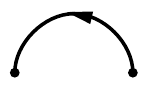}} \mapsto    \raisebox{-6pt}{\includegraphics[scale = 0.4]{n+.pdf}}\\
          \raisebox{-9pt}{\includegraphics[scale = 0.4]{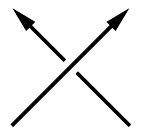}} &\mapsto  q^2 \raisebox{-9pt}{\includegraphics[scale = 0.4]{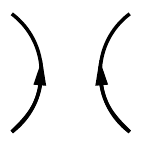}} - q^3 \raisebox{-9pt}{\includegraphics[scale = 0.4]{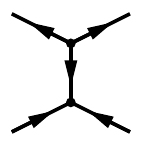}} 
          &&\raisebox{-9pt}{\includegraphics[scale = 0.4]{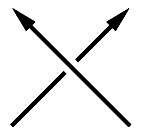}} \mapsto   q^{-2} \raisebox{-9pt}{\includegraphics[scale = 0.4]{Vertbarup.pdf}} - q^{-3} \raisebox{-9pt}{\includegraphics[scale = 0.4]{Fcross2.pdf}}.\\       
    \end{align*}
To make sure that Theorem ~\ref{genrel} applies we must show that the relations of $\mathbf{OTa}$  still hold when taken through $\phi$. We will demonstrate why relations $(e_+)$, $(c)$ and $(b)$ in Theorem ~\ref{tanglerelations} hold in $\mathbf{LWeb}$. The other relations can be verified in a similar manner.

For the relation $(e_+)$, we wish to show that $\displaystyle{\phi \left(\raisebox{-7pt}{\includegraphics[height=.25in]{f1.pdf}} \right)  =  \phi\left( \raisebox{-7pt}{\includegraphics[height=.25in]{f2.pdf}} \right)}$. By the definition of $\phi$ note that we have the following:
$$ \phi \left(\raisebox{-15pt}{\includegraphics[height=.5in]{f1.pdf}} \right)  = q^2\raisebox{-15pt}{\includegraphics[height=.5in]{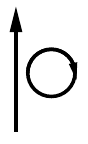}}  -q^3 \raisebox{-15pt}{\includegraphics[height=.5in]{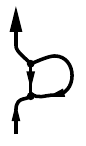}}= q^2\raisebox{-15pt}{\includegraphics[height=.5in]{Eq1-2.pdf}}  -q^3 \raisebox{-15pt}{\includegraphics[height=.5in]{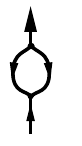}} .$$

By the relations imposed on the morphisms of $\mathbf{LWeb}$ from quotienting by $N_1(\epsilon, \nu)$ and $N_2(\epsilon, \nu)$, we have
$$\raisebox{-15pt}{\includegraphics[height=.5in]{Eq1-3.pdf}} = (q+q^{-1}) \raisebox{-15pt}{\includegraphics[height=.5in]{a2+.pdf}} \text{ and } \raisebox{-15pt}{\includegraphics[height=.5in]{Eq1-2.pdf}} = (q^2 +1 +q^{-2}) \raisebox{-15pt}{\includegraphics[height=.5in]{a2+.pdf}}.
$$
Hence,
\[
q^2\raisebox{-15pt}{\includegraphics[height=.5in]{Eq1-2.pdf}}  -q^3 \raisebox{-15pt}{\includegraphics[height=.5in]{Eq1-3.pdf}} = q^2(q^2 + 1 + q^{-1})\raisebox{-15pt}{\includegraphics[height=.5in]{a2+.pdf}} -q^{3} (q+q^{-1}) \raisebox{-15pt}{\includegraphics[height=.5in]{a2+.pdf}}
= \raisebox{-15pt}{\includegraphics[height=.5in]{a2+.pdf}} = \phi \left( \raisebox{-15pt}{\includegraphics[height=.5in]{a2+.pdf}}\right).
\]

Therefore, $\phi$ preserves relation $(e_+)$. To prove that $\phi$ preserves relation $(c)$, we need to show that 
$$\phi(X_+ \circ X_-) = \phi(X_- \circ X_+)  = \phi(I_+ \otimes I_+) = I_+ \otimes I_+ = I_{(+, +)}.$$
By definition, we require that $\phi$ preserves composition. Hence,
\begin{align*}
\phi(X_+ \circ X_-) &= \phi(X_+) \circ \phi( X_-)\\
&= \left(q^{2}\mychar{Vertbarup.pdf} - q^3\mychar{Fcross2.pdf} \right)\circ \left(q^{-2}\mychar{Vertbarup.pdf} - q^{-3}\mychar{Fcross2.pdf} \right)\\
&= \mychar{Vertbarup.pdf}\circ\mychar{Vertbarup.pdf} - q\mychar{Fcross2.pdf}\circ\mychar{Vertbarup.pdf} - q^{-1}\mychar{Vertbarup.pdf}\circ\mychar{Fcross2.pdf} + \mychar{Fcross2.pdf}\circ  \mychar{Fcross2.pdf}\\
&= \mychar{Vertbarup.pdf} - q\mychar{Fcross2.pdf} - q^{-1}\mychar{Fcross2.pdf} + \raisebox{-7pt}{\includegraphics[height=.35in]{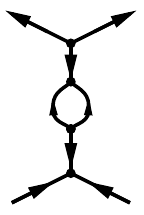}}\\
&= \mychar{Vertbarup.pdf} - (q+q^{-1})\mychar{Fcross2.pdf} + (q+q^{-1})\mychar{Fcross2.pdf}\\
&= \mychar{Vertbarup.pdf} = I_{(+, +)}.
\end{align*}
The proof showing that $\phi(X_- \circ X_+) =  I_{(+, +)}$ is done in a similar way, and thus we omit it to avoid repetition.

To verify that the functor $\phi$ preserves relation $(d)$, we need  to show the following:
$$\phi(X_+ \otimes I_+)\circ\phi(I_+ \otimes X_+)\circ \phi(X_+ \otimes I_+) = \phi(I_+ \otimes X_+)\circ \phi(X_+ \otimes I_+)\circ\phi(I_+ \otimes X_+). $$

By using the definition of the functor $\phi$ on the generators of the category $\mathbf{OTa}$, we obtain the following: 
\begin{align*}
	&\phi(X_+ \otimes I_+)\circ\phi(I_+ \otimes X_+)\circ \phi(X_+ \otimes I_+)\\
	&= \left(q^{2}\mychar{a2+}\mychar{a2+}\mychar{a2+} - q^3\mychar{Fcross2.pdf}\mychar{a2+} \right)\circ \left(q^{2}\mychar{a2+}\mychar{a2+}\mychar{a2+} - q^{3}\mychar{a2+}\mychar{Fcross2.pdf} \right) \left(q^{2}\mychar{a2+}\mychar{a2+}\mychar{a2+} - q^3\mychar{Fcross2.pdf}\mychar{a2+} \right)\\
	&= q^{6}\mychar{a2+}\mychar{a2+}\mychar{a2+} - q^7\mychar{Fcross2.pdf}\mychar{a2+} - q^7\mychar{a2+}\mychar{Fcross2.pdf} + q^8 \raisebox{-6pt}{\includegraphics[scale = 0.4]{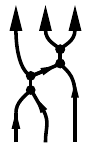}}
	-  q^7\mychar{Fcross2.pdf}\mychar{a2+} + q^8\raisebox{-6pt}{\includegraphics[scale = 0.4]{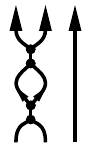}}\\
	&+q^8 \raisebox{-6pt}{\includegraphics[scale = 0.4]{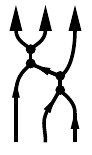}}-
	q^9\raisebox{-6pt}{\includegraphics[scale = 0.4]{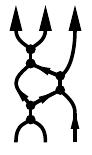}}. 
\end{align*}

Similarly, we have:
\begin{align*}
	& \phi(I_+ \otimes X_+)\circ \phi(X_+ \otimes I_+)\circ\phi(I_+ \otimes X_+)\\
	&= q^{6}\mychar{a2+}\mychar{a2+}\mychar{a2+} - q^7\mychar{a2+}\mychar{Fcross2.pdf} - q^7\mychar{Fcross2.pdf}\mychar{a2+} + q^8 \raisebox{-6pt}{\includegraphics[scale = 0.4]{Eq3-2.pdf}}
	-  q^7\mychar{a2+}\mychar{Fcross2.pdf} + q^8\raisebox{-6pt}{\includegraphics[scale = 0.4]{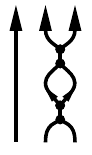}}\\
	& + q^8 \raisebox{-6pt}{\includegraphics[scale = 0.4]{Eq3-2r.pdf}}
	- q^9\raisebox{-6pt}{\includegraphics[scale = 0.4]{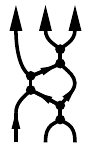}}.
\end{align*}

Now subtracting $\phi(I_+ \otimes X_+)\circ \phi(X_+ \otimes I_+)\circ\phi(I_+ \otimes X_+)$ from  $\phi(X_+ \otimes I_+)\circ\phi(I_+ \otimes X_+)\circ \phi(X_+ \otimes I_+)$, we get:
$$-  q^7\mychar{Fcross2.pdf}\mychar{a2+} + q^8\raisebox{-6pt}{\includegraphics[scale = 0.4]{Eq3-3.pdf}}-  q^9\raisebox{-6pt}{\includegraphics[scale = 0.4]{Eq3-1.pdf}} - \left( 	-  q^7\mychar{a2+}\mychar{Fcross2.pdf} + q^8\raisebox{-6pt}{\includegraphics[scale = 0.4]{Eq3-3r.pdf}}
- q^9\raisebox{-6pt}{\includegraphics[scale = 0.4]{Eq3-1r.pdf}} \right).$$
By using the digon relation imposed on $\mathbf{LWeb}$, we get 
  $$-  q^7\mychar{Fcross2.pdf}\mychar{a2+} + q^8\raisebox{-6pt}{\includegraphics[scale = 0.4]{Eq3-3.pdf}} = -  q^7\mychar{Fcross2.pdf}\mychar{a2+} + q^8(q + q^{-1})\mychar{Fcross2.pdf}\mychar{a2+} = q^9\mychar{Fcross2.pdf}\mychar{a2+}.$$
  Similarly, 
 $$-  q^7\mychar{a2+}\mychar{Fcross2.pdf} + q^8\raisebox{-6pt}{\includegraphics[scale = 0.4]{Eq3-3r.pdf}} = -  q^7\mychar{a2+}\mychar{Fcross2.pdf} + q^8(q + q^{-1})\mychar{a2+}\mychar{Fcross2.pdf} = q^9\mychar{a2+}\mychar{Fcross2.pdf}.$$
 
 This then simplifies the difference $\phi(X_+ \otimes I_+)\circ\phi(I_+ \otimes X_+)\circ \phi(X_+ \otimes I_+) - \phi(I_+ \otimes X_+)\circ \phi(X_+ \otimes I_+)\circ\phi(I_+ \otimes X_+)$ to 
$$ q^9\mychar{Fcross2.pdf}\mychar{a2+} -  q^9\raisebox{-6pt}{\includegraphics[scale = 0.4]{Eq3-1.pdf}} -	 q^9\mychar{a2+}\mychar{Fcross2.pdf} 
+ q^9\raisebox{-6pt}{\includegraphics[scale = 0.4]{Eq3-1r.pdf}}.  \\ $$
Then, since
$$\raisebox{-16pt}{\includegraphics[height = 1.5cm]{Eq3-1.pdf}} \text{ is isotopic to } \raisebox{-16pt}{\includegraphics[height = 1.5cm]{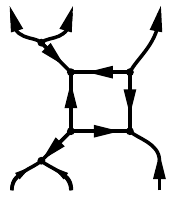}},$$
we can use the relation imposed on $\mathbf{Lweb}$ by quotienting the morphisms by $N_3(\epsilon, \nu)$, to obtain the following:
$$q^9\raisebox{-9pt}{\includegraphics[scale = 0.4]{Eq3-1.pdf}} = q^9\mychar{Fcross2.pdf}\mychar{a2+} +  q^9\raisebox{-9pt}{\includegraphics[scale = 0.3]{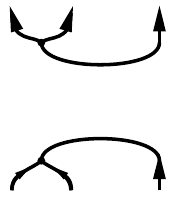}},$$
$$q^9\raisebox{-9pt}{\includegraphics[scale = 0.4]{Eq3-1r.pdf}} = q^9\mychar{a2+}\mychar{Fcross2.pdf} +  q^9\raisebox{-9pt}{\includegraphics[scale = 0.3]{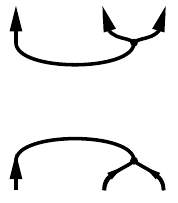}}.$$

Hence, 
\begin{align*}
	&\phi(X_+ \otimes I_+)\circ\phi(I_+ \otimes X_+)\circ \phi(X_+ \otimes I_+) - \phi(I_+ \otimes X_+)\circ \phi(X_+ \otimes I_+)\circ\phi(I_+ \otimes X_+) \\
	&= q^9\mychar{Fcross2.pdf}\mychar{a2+} -  q^9\raisebox{-9pt}{\includegraphics[scale = 0.4]{Eq3-1.pdf}} -	 q^9\mychar{a2+}\mychar{Fcross2.pdf} 
	+ q^9\raisebox{-9pt}{\includegraphics[scale = 0.4]{Eq3-1r.pdf}}  \\
	&= q^9\mychar{Fcross2.pdf}\mychar{a2+} -   q^9\mychar{Fcross2.pdf}\mychar{a2+} -  q^9\raisebox{-6pt}{\includegraphics[scale = 0.3]{Eq3-5.pdf}}-	 q^9\mychar{a2+}\mychar{Fcross2.pdf} 
	+  q^9\mychar{a2+}\mychar{Fcross2.pdf} +  q^9\raisebox{-9pt}{\includegraphics[scale = 0.3]{Eq3-5r.pdf}}  \\
	&= q^9\raisebox{-9pt}{\includegraphics[scale = 0.3]{Eq3-5.pdf}} - q^9\raisebox{-9pt}{\includegraphics[scale = 0.3]{Eq3-5r.pdf}} = 0,
\end{align*}
where the last equality holds since the two webs are planar isotopic. Therefore, the functor $\phi$ preserves relation (d).

It remains to show that $F = \phi$. However, first we will modify notation. Let $T$ be any tangle diagram, we may request $T$ to be in general position. Then, if we order the crossings $(v_1, \ldots, v_n)$ in $T$ from top to bottom, there is a unique resolution of $T$ for any element of $\{0,1\}^n$. More explicitly, any element $(a_1, \ldots, a_n) \in \{0,1\}^n$ corresponds to the resolution of $T$ that takes the $jth$ crossing to $\mychar{Vertbarup.pdf}$ if $a_j = 0$ and to $\mychar{Fcross2.pdf}$ if $a_j = 1$. If we denote this resolution as $\Gamma_{(a_1, \ldots, a_n)}$ and similarly, we denote the coefficient associated with this resolution as $b_{(a_1, \ldots, a_n)}$, where the coefficient is determined by the rules in Fig.~\ref{rule}, then we may rewrite 
$$F(\Gamma) = \sum_{\alpha \in \{0,1\}^n} b_\alpha \Gamma_\alpha.$$

We need to show that $F(T) = \phi(T)$ for any $T \in \mathbf{OTa}$, and we prove this by induction on the number of crossings of $T$.

\noindent
\textbf{Base Case:} If $T$ has zero crossings, then both $F(T) = T$ and $\phi(T) = T$, since we can regard $T$ as a web with no vertices.

\noindent
\textbf{Inductive Step:} Suppose that $F(A) = \phi(A)$ for all $A \in \mathbf{OTa}$ that have less than $n$ crossings. Let $T$ be any tangle diagram with $n$ crossings. Then by Lemma ~\ref{generalform}, there exists generators $W_1, W_2,\ldots, W_m$, such that
\[ T= (I_{v_1}\otimes W_1 \otimes I_{u_1})\circ (I_{v_2}\otimes W_2 \otimes I_{u_2}) \circ \ldots \circ (I_{v_n}\otimes W_n \otimes I_{u_m}),\]
where the $v_i$'s and $u_i's$ are objects in $\mathbf{OTa}$. Let $k$ be the lowest number such that $W_k = X_+$ or $X_-$. Then we can write
\[ T = T' \circ (I_{v_{k}}\otimes W_k \otimes I_{u_{k}}) \circ T'',\]
where 
\[T' := (I_{v_1}\otimes W_1 \otimes I_{u_1})\circ \ldots \circ  (I_{v_{k-1}}\otimes W_{k-1} \otimes I_{u_{k-1}})\]
and 
\[T'' := (I_{v_{k+1}}\otimes W_{k+1} \otimes I_{u_{k+1}}) \circ \ldots \circ (I_{v_m}\otimes W_m \otimes I_{u_m}).\]
 That is, $T'$ has no crossings and $T''$ has $n-1$ crossings. Then we have 
\begin{align*}
F(T) &= \sum_{\alpha \in \{0,1\}^n} b_\alpha \Gamma_\alpha\\
&= \sum_{\alpha \in \{0\}\times\{0,1\}^{n-1}} b_\alpha \Gamma_\alpha + \sum_{\alpha \in \{1\}\times\{0,1\}^{n-1}} b_\alpha \Gamma_\alpha \\
&= \sum_{\alpha \in \{0,1\}^{n-1}} b_{(0\otimes \alpha)} \Gamma_{(0\otimes\alpha)} + \sum_{\alpha \in \{0,1\}^{n-1}} b_{(1 \otimes \alpha)} \Gamma_{(1 \otimes \alpha)}.
\end{align*}
Without loss of generality, let $W_k = X_+$. Then we have 
\begin{align*}
F(T) & =\sum_{\alpha \in \{0,1\}^{n-1}} b_{(0\otimes \alpha)} \Gamma_{(0\otimes\alpha)} + \sum_{\alpha \in \{0,1\}^{n-1}} b_{(1 \otimes \alpha)} \Gamma_{(1 \otimes \alpha)}\\
&= \sum_{\alpha \in \{0,1\}^{n-1}} q^2b_{\alpha}(T' \circ (I_{v_{k}}\otimes \mychar{Vertbarup.pdf} \otimes I_{u_{k}}) \circ \Gamma'')_{(0 \otimes \alpha)} \\
&- \sum_{\alpha \in \{0,1\}^{n-1}} q^3b_{ \alpha} (T' \circ (I_{v_{k}}\otimes \mychar{Fcross2.pdf}\otimes I_{u_{k}}) \circ \Gamma'')_{(1 \otimes \alpha)}\\
&= \sum_{\alpha \in \{0,1\}^{n-1}} q^2b_{\alpha}T' \circ (I_{v_{k}}\otimes \mychar{Vertbarup.pdf} \otimes I_{u_{k}}) \circ \Gamma''_{\alpha} \\
&- \sum_{\alpha \in \{0,1\}^{n-1}} q^3b_{\alpha}T' \circ (I_{v_{k}}\otimes \mychar{Fcross2.pdf} \otimes I_{u_{k}}) \circ \Gamma''_{\alpha},
\end{align*}
where $\Gamma''_{\alpha}$ are the $2^{n-1}$ resolutions of the tangle $T''$ containing $n-1$ crossings.

Using that composition is bilinear in $\mathbf{LWeb}$, we can rewrite the above sum as follows:
\begin{align*}
F(T) &= T' \circ(I_{v_{k}}\otimes q^2\mychar{Vertbarup.pdf} \otimes I_{u_{k}}) \circ \left( \sum_{\alpha \in \{0,1\}^{n-1}} b_{\alpha}  \Gamma''_{\alpha} \right)\\
&- T' \circ(I_{v_{k}}\otimes q^3\mychar{Fcross2.pdf} \otimes I_{u_{k}}) \circ \left(\sum_{\alpha \in \{0,1\}^{n-1}} b_{\alpha}  \Gamma''_{\alpha} \right)\\
&= T' \circ(I_{v_{k}}\otimes (q^2\mychar{Vertbarup.pdf} - q^3\mychar{Fcross2.pdf}) \otimes I_{u_{k}}) \circ \left( \sum_{\alpha \in \{0,1\}^{n-1}} b_{\alpha}  \Gamma''_{\alpha} \right)\\
&= F(T') \circ(I_{v_{k}}\otimes \phi(X_+) \otimes I_{u_{k}}) \circ F(T'').
\end{align*}
By the base case and inductive hypothesis, $\phi(T') = F(T')$ and  $\phi(T'') = F(T'')$. Moreover, by the functoriality of $\phi$, we obtain the following: 
$$F(T) =\phi(T') \circ (I_{v_{k}}\otimes \phi(X_+) \otimes I_{u_{k}}) \circ \phi(T'') = \phi(T' \circ (I_{v_{k}}\otimes X_+ \otimes I_{u_{k}}) \circ T'') = \phi(T).$$ 
Thus, the inductive step holds so by the principle of induction, $F = \phi$ and thus $F$ is a tensor-product-preserving functor. 
\end{proof}

\section*{Acknowledgments}
The author is grateful for the mentorship provided by Prof. Carmen Caprau during the research and writing of this paper, and for the NSF support through the RUI grant, DMS-2204386.





\end{document}